\documentclass[11pt]{article}

\usepackage{amsmath,amsfonts,amssymb}
\usepackage{bbm}

\usepackage{lipsum}

\usepackage{xcolor}
\usepackage{ulem}

\newtheorem{defn}{Definition}[section]
\newtheorem{theo}[defn]{Theorem}
\newtheorem{lem}[defn]{Lemma}
\newtheorem{prop}[defn]{Proposition}
\newtheorem{cor}[defn]{Corollary}
\newtheorem{rem}[defn]{Remark}
\newtheorem{exam}[defn]{Example}

\newenvironment{proof}{{\bf Proof }}{{\vskip 0.1cm \hfill$\Box$}}

\def\N {{\mathbb N}}

\def\R {{\mathbb R}}
\def\E{{\mathbb E}}
\def\P{{\mathbb P}}
\def\M{{\mathbb M}}

\begin{document}


\noindent
{{\Large\bf  Existence and uniqueness of (infinitesimally) invariant measures for second order partial differential operators on Euclidean space}
{\footnote{The research of Haesung Lee was supported by Basic Science Research Program through the National Research Foundation of Korea (NRF) funded by the Ministry of Education (2020R1A6A3A01096151). The research of Gerald Trutnau was supported by the Basic Science Research Program through the National Research Foundation of Korea (NRF) funded by the Ministry of Education (2017R1D1A1B03035632).}} \\ \\
\bigskip
\noindent
{\bf Haesung Lee}, {\bf Gerald Trutnau} \\
\noindent
{\small{\bf Abstract.}  
We consider a locally uniformly strictly elliptic second order partial differential operator in $\R^d$, $d\ge 2$, with low regularity assumptions on its coefficients, as well as an associated Hunt process and semigroup. The Hunt process is known to solve a corresponding stochastic differential equation that is pathwise unique. In this situation, we study the relation of invariance, infinitesimal invariance, recurrence, transience, conservativeness and $L^r$-uniqueness, and present sufficient conditions for non-existence of finite infinitesimally invariant measures as well as finite invariant measures.
Our main result is that recurrence implies uniqueness of infinitesimally invariant measures, as well as existence and uniqueness of invariant measures, both in subclasses of locally finite measures. We can hence make in particular use of various explicit analytic criteria for recurrence that have been previously developed in the context of (generalized) Dirichlet forms and
present diverse examples and counterexamples for uniqueness of infinitesimally invariant, as well as invariant measures and 
an example where $L^1$-uniqueness fails for one infinitesimally invariant measure but holds for another and pathwise uniqueness holds.
Furthermore, we illustrate how our results can be applied to related work and vice versa.


\section{Introduction}
Invariant measures have been broadly studied for a long time in the area of dynamical systems theory, semigroup theory and stochastic analysis. In this article, we deal with invariant measures on Euclidean space for a semigroup or a continuous-time Markov process associated with a second order partial differential operator. If under certain assumptions a semigroup or a Markov process admit a finite invariant measure, the asymptotic behavior of the semigroup or the transition semigroup as $t \rightarrow \infty$ can explicitly be characterized in contrast to the case of infinite invariant measures (e.g. \cite[Theorem 3.4]{has}, \cite[Theorem 4.3]{kha12}, \cite[Theorem 4.2.1]{DPZB}, \cite[Theorem 8.1.24]{LB07}).
For this reason, infinite invariant measures do not seem to have as much interest as finite invariant measures. However, we remark that studying infinite invariant measures is also valuable since these are strongly related to $L^1$-uniqueness (\cite[Corollary 2.2]{St99}), uniqueness of the martingale problem (\cite[Section 3.3.2]{LST20}) as well as a characterization for an operator core in an $L^1$-space (Remark \ref{equivconinv(iii)}). Therefore, in this article, invariant measures which are also allowed to be infinite (cf. \cite{Eb}) will be treated. \\
Consider a time-homogeneous It\^{o}-SDE on $\R^d$ with $d \geq 2$ 
\begin{equation} \label{underlyingsde}
X_t =x+\int_0^t \sigma(X_s)dW_s + \int_0^t \mathbf{G}(X_s)ds,  \quad 0 \leq t < \infty,
\end{equation}
where $x \in \R^d$, $\sigma=(\sigma_{ij})_{1 \leq i,j \leq d}$ is a matrix of functions, $\mathbf{G}=(g_1, \ldots, g_d)$ is a vector field on $\R^d$, and the corresponding partial differential operator $(L, C_0^{\infty}(\R^d))$ defined by
\begin{equation} \label{the generator}
Lf = \frac12 \sum_{i,j=1}^d a_{ij} \partial_{ij} f+\sum_{i=1}^d g_i \partial_i f, \qquad f \in C_0^{\infty}(\R^d), \quad A=(a_{ij})_{1 \leq i,j \leq d}=\sigma \sigma^T.
\end{equation}
In \cite[Theorem 3.7]{kha12}, under the assumption that the components of $\sigma$ and $\mathbf{G}$ are locally Lipschitz continuous and that the corresponding unique strong solution $(X^x_t)_{t \geq 0}$ to \eqref{underlyingsde} with respect to the standard Brownian motion $(\Omega, \mathcal{F}, \P, (W_t)_{t \geq 0})$ is regular for at least one point of $\R^d$ (for instance if it is generally non-explosive), the existence of an invariant probability measure $\nu$ on $(\R^d, \mathcal{B}(\R^d))$ for $(X^x_t)_{t \geq 0}$, i.e. 
$$
\int_{\R^d} \P(X^x_t \in A) \nu(dx) = \nu(A), \quad \forall A \in \mathcal{B}(\R^d),
$$
is shown under the assumption that  there exists a positive function $V \in C^2(\R^d)$ such that 
$$
\sup_{\R^d \setminus B_R}LV \longrightarrow -\infty \quad \text{as $R \rightarrow \infty$}. 
$$
The main ingredient to show the result above is the tightness of measures (see \cite[Theorem 3.1]{kha12}), which goes back to Krylov-Bogolyobov's work. The method to derive the result above is also applicable in the case where the components of $A$ and $\mathbf{G}$ are locally H\"{o}lder continuous and $A$ is locally uniformly strictly elliptic (see \cite[Theorem 6.3]{MPW02} and \cite[Theorem 8.1.20]{LB07}). Furthermore, uniqueness of invariant measures in a subclass of probability measures can be shown by using the strong Feller property and the irreducibility of the semigroup (see \cite[Theorem 4.2.1]{DPZB} and \cite[Theorem 8.1.15]{LB07}).\\
In \cite[Theorem 1.5]{St99}, assuming the existence of an infinitesimally invariant measure $\mu$ for $(L, C_0^{\infty}(\R^d))$ (see Definition \ref{definvaran}) with nice density with respect to Lebesgue measure and more general local regularity assumptions on the coefficients than {\bf (H)} of Section \ref{section1}, it is shown that there exists an $L^1(\R^d, \mu)$-closed extension $(\overline{L}^{\mu}, D(\overline{L}^{\mu}))$ of $(L, C_0^{\infty}(\R^d))$ which generates a sub-Markovian $C_0$-semigroup of contractions $(\overline{T}^{\mu}_t)_{t>0}$ on $L^1(\R^d, \mu)$. $(\overline{T}^{\mu}_t)_{t>0}$ further induces a sub-Markovian $C_0$-semigroup of contractions $(T^{\mu}_t)_{t>0}$ on $L^r(\R^d, \mu)$, $r \in [1, \infty)$, and a sub-Markovian semigroup on $L^{\infty}(\R^d, \mu)$, also denoted $(T^{\mu}_t)_{t>0}$ (for the details, see Section \ref{section1}). 
Here, we consider assumption {\bf (H)} of Section \ref{section1} on the coefficients of $L$, which implies the existence of a nice density and that \cite[Theorem 1.5]{St99} is applicable, i.e. the existence of $(\overline{T}^{\mu}_t)_{t>0}$ and $(T^{\mu}_t)_{t>0}$ as above (cf. Remark \ref{infini and H}).\\
The existence of a {\it finite} infinitesimally invariant measure $\mu$ for $(L, C_0^{\infty}(\R^d))$ in case $A$ and $\mathbf{G}$ satisfy {\bf (H)} is shown in \cite[Theorem 1.2]{BR99} under the assumption that there exists a positive function $V \in C^2(\R^d)$ with
\begin{equation} \label{lapynovin}
 LV(x) \longrightarrow -\infty, \; \;\; V(x) \longrightarrow \infty \quad \text{ as $\|x\| \rightarrow \infty$}.
\end{equation}
Moreover, in the situation of \cite[Theorem 1.2]{BR99} stated above, it follows that $\mu$ is an invariant measure for $(T^{\mu}_t)_{t>0}$ (see Definition \ref{definvaran2}). The result \cite[Theorem 1.2]{BR99} is improved in \cite[Theorem 2]{BRS} where condition \eqref{lapynovin} is relaxed to the following condition: there exists a positive function $V\in C^2(\R^d)$ and constants $K>0$, $N_0 \in \N$ such that
\begin{equation}  \label{lapynovinrelax}
LV \leq -K, \quad \text{ on } \R^d \setminus \overline{B}_{N_0}, \;\;\;V(x) \longrightarrow \infty \;\; \text{ as $\|x\| \rightarrow \infty$}.
\end{equation}
By \cite[Corollary 5.4.3 and Proposition 5.3.6]{BKRS}, one can check that the uniqueness of infinitesimally invariant measures for $(L, C_0^{\infty}(\R^d))$ and invariant measures for $(T_t^{\mu})_{t>0}$ in a subclass of probability measures holds under assumption \eqref{lapynovinrelax}. More generally, through Doob's Theorem (\cite[Theorem 4.2.1]{DPZB}), it is derived in \cite[Proposition 3.47]{LST20}, that if there exist constants $K>0$, $N_0 \in \N$ and a positive function $V\in C^2(\R^d \setminus \overline{B}_{N_0}) \cap C(\R^d)$ such that \eqref{lapynovinrelax} holds,
then every (possibly infinite) invariant measure for the Hunt process $\M$ associated to $(T^{\mu}_t)_{t>0}$ (defined as in \cite[Definition 3.43]{LST20}) is finite and $\mu$ is  the unique invariant measure for $\M$ up to constant multiples. However \cite[Theorem 4.2.1]{DPZB} and consequently \cite[Proposition 3.47]{LST20} do not say anything about uniqueness of infinitesimally invariant measures. On the other hand \cite[Proposition 3.47]{LST20} holds under more general conditions than {\bf (H)}. \\
The existence and uniqueness of invariant measures for a strong Markov process in the class of $\sigma$-finite measures is investigated in \cite{has} and \cite{MT}, by assuming that the strong Markov process satisfies diverse technical conditions, where it turns out that besides the technical conditions, recurrence of the strong Markov process is a sufficient condition to obtain existence and uniqueness of invariant measures. In particular, by probabilistic means, the results of \cite{has} apply to the specific case of \cite{Bha}, where the diffusion coefficient is continuous on $\R^d$ and the drift coefficient is locally bounded on $\R^d$. 
However, it may not be easy to check whether (besides recurrence) some of the technical conditions of \cite{has} and \cite{MT} are fulfilled 
in general and in particular in our setting that we are now going to introduce. We consider $L$ as in \eqref{the generator} with coefficients $A=(a_{ij})_{1 \leq i,j \leq d}$ and $\mathbf{G}$ satisfying assumption {\bf (H)} of Section \ref{section1}. In this setting, the following can be checked from \cite{LT18} (see also \cite{LST20} for a complete exposition under more general assumptions, or Section \ref{section1}):
\begin{itemize}
\item there exists an infinitesimally invariant measure $\mu=\rho\,dx$ for $(L, C_0^{\infty}(\R^d))$ (see \eqref {infinvalfir} and Remark \ref{infini and H}). Such a measure is fixed throughout this article and we may also allow that $\mu$ is  explicitly chosen according to assumption {\bf (H$^{\prime}$)} of Section \ref{section1} which implies {\bf (H)}
\item there exists $(T^{\mu}_t)_{t>0}$ as described in the paragraph right before \eqref{lapynovin}
\item $(T^{\mu}_t)_{t>0}$ has a regularized version $(P^{\mu}_t)_{t>0}$, which is the transition semigroup of a Hunt process 
 $\M$, constructed in \cite[Theorem 3.12]{LT18} (see also \cite[Theorem 3.11]{LST20})
\item $(T^{\mu}_t)_{t>0}$ is strictly irreducible and $\M$ is irreducible in the probabilistic sense by \cite[Corollary 4.8]{LT18} (see also \cite[Proposition 2.39]{LST20})
\item we have the dichotomy that $\M$ is either transient or recurrent in the probabilistic sense by \cite[Proposition 4.9]{LT18} (see also \cite[Theorem 3.38(iv)]{LST20})
\item various sufficient conditions for  non-explosion of $\M$ can be found in  \cite[Section 3.2.1]{LST20}. In particular $\M$ is non-explosive, if and only if $(T^{\mu}_t)_{t>0}$ is conservative
\item if $\sigma=(\sigma_{ij})_{1 \leq i,j \leq d}$ is any not necessarily symmetric matrix of functions $\sigma_{ij} \in H^{1,p}_{loc}(\R^d) \cap C(\R^d)$ for some $p\in (d,\infty)$, $d\ge 2$, $1\le i,j\le d$, with $A=(a_{ij})_{1 \leq i,j \leq d}=\sigma \sigma^T$,  and $\M$ is non-explosive, then the Hunt process $\M$ weakly solves the SDE \eqref{underlyingsde} for which pathwise uniqueness and diverse other properties apart from the ones already mentioned are known to hold (cf. \cite[Theorem 3.52]{LST20}).
\end{itemize}
In the above described setting, our main result consists of Theorems \ref{preinvlemma} and \ref{existuniquenessinv} which state that recurrence implies uniqueness of infinitesimally invariant measures (Theorem \ref{preinvlemma}), as well as existence and uniqueness of invariant measures (Theorem \ref{existuniquenessinv}), both in subclasses of locally finite measures. But, we note that the converse does not hold (see Remark \ref{dxunique}). Theorem \ref{preinvlemma} can be used to reconfirm and extend existing literature by using a different method of proof (here we ignore the slight technical differences between assumption {\bf (H)} and \cite[(5.2.1)]{BKRS}, cf. \cite[Remark  2.40]{LST20}). For instance, if one searches for uniqueness of infinitesimally invariant measures in a subclass of probability measures, it is known from \cite[Theorems 5.3.1, 5.4.1]{BKRS} that a sufficient condition under assumption {\bf (H)} is the existence of an infinitesimally invariant measure that is a probability invariant measure. A combination of Proposition \ref{fininvariant} and Theorem \ref{preinvlemma} immediately gives that result with the uniqueness of infinitesimally invariant measures being now even in a subclass of locally finite measures. Here, one has to note that in the situation, where infinitesimally invariant measures $\mu$ are a priori assumed to be finite as in \cite{BKRS}, recurrence of $(T^{\mu}_t)_{t>0}$ is 
equivalent to invariance of $\mu$ with respect to $(T^{\mu}_t)_{t>0}$ by Proposition \ref{fininvariant} and Corollary \ref{existinvariancemea}. In this regard, by \cite[Theorem 5.3.1]{BKRS}, Theorem \ref{preinvlemma} provides a full generalization of \cite[Theorem 5.4.1]{BKRS} from a subclass of probability measures to a subclass of locally finite measures. Moreover, recurrence implies $L^1$-uniqueness with respect to a unique (possibly infinite) infinitesimally invariant and invariant measure (by our main result, Corollary 3.8 and \cite[Corollary 2.2]{St99}). 
A further result is the characterization of absence of finite invariant measures (Corollary \ref{carinva}) and a sufficient condition for absence of finite infinitesimally invariant meaures (Proposition \ref{prop1.22}(ii)). Concrete examples for the latter appear  in Examples \ref{exampleinva}, \ref{exampleinva2}(ii), \ref{exam:3.1} and \ref{dualnotconser}. \\
To find an invariant measure, we focus on showing conservativeness for the dual semigroup $(T'^{,\mu}_t)_{t>0}$ of $(T^{\mu}_t)_{t>0}$ which is equivalent to the fact that $\mu$ is an invariant measure for $(T^{\mu}_t)_{t>0}$ by Lemma \ref{equicondi32}(i). Indeed, since recurrence of $(T^{\mu}_t)_{t>0}$ (see Definition \ref{defrectra}) implies the recurrence of the dual semigroup of $(T^{\mu}_t)_{t>0}$ (see \cite[Proposition 2.5]{BCR} and Theorem \ref{cosemirec} for an alternative proof), as well as recurrence implies conservativeness (Corollary \ref{existinvariancemea}), it follows that recurrence of $(T^{\mu}_t)_{t>0}$ is a sufficient condition for $\mu$ to be an invariant measure for $(T^{\mu}_t)_{t>0}$ by Lemma \ref{equicondi32}(i). We show the uniqueness of invariant measures through the uniqueness of infinitesimally invariant measures (see Theorem \ref{existuniquenessinv} and its proof). We use a property of excessive functions (\cite[Proposition 11(b)]{GT2}) for the Hunt process $\M$ constructed in \cite[Theorem 3.12]{LT18}, so that recurrence of $(T^{\mu}_t)_{t>0}$  turns out to be not only a sufficient criterion to obtain existence of an invariant measure but also to obtain uniqueness of (infinitesimally) invariant measures (see proof of Theorem \ref{preinvlemma}). This is mainly done in Section \ref{invrec}, whereas existence of an infinitesimally invariant measure follows from assumption {\bf (H)} (or {\bf (H$^{\prime}$)}).\\
The advantage of our main result is that we can use various previously developed explicit criteria for recurrence in the frame of (generalized) Dirichlet forms (\cite{GT2}, \cite{LT18}, \cite{LST20}, \cite{Sturm98}, \cite{Stu1}, \cite{Stu2}) to obtain existence and uniqueness of (infinitesimally) invariant measures. These criteria depend only on analytic quantities (e.g. the coefficients of $L$ or volume growth, etc.) and not on the stochastic process. 
For instance, by Remark  \ref{explicitcritinv}(i) (which states a consequence of an explicit recurrence result from \cite{LST20} and our main result), if there exist $N_0 \in \N$ and a positive function $V\in C^2(\R^d \setminus \overline{B}_{N_0}) \cap C(\R^d)$ such that
\begin{equation*}  
\label{lapynovrecinv}
LV \leq 0 \quad \text{ on } \R^d \setminus \overline{B}_{N_0}, \;\;\;  V(x) \longrightarrow \infty \;\; \text{ as $\|x\| \rightarrow \infty$},
\end{equation*}
then $(T^{\mu}_t)_{t>0}$ is recurrent and $\mu$ is the unique (possibly infinite) infinitesimally invariant measure  for $(L, C_0^{\infty}(\R^d))$ and the unique invariant measure for $(T^{\mu}_t)_{t>0}$ (both up to constant multiples). Or in the case $d=2$,  an arbitrary growth condition on  the diffusion coefficient can be allowed to get recurrence of $(T^{\mu}_t)_{t>0}$ and hence uniqueness of infinitesimally invariant measures for $(L, C_0^{\infty}(\R^d))$ and uniqueness of invariant measures for $(T^{\mu}_t)_{t>0}$, if the difference between the maximum and minimum eigenvalue of functions is sufficiently small and the drift $\mathbf{G}$ is suitable 
(see \eqref{eq:3.2.1.20} in Remark \ref{explicitcritinv}(ii) which is again a consequence of results from \cite{LST20} and our main result). 
Further explicit conditions, including volume growth and integrability conditions, for existence and uniqueness of a prescribed (cf. \cite[2.2.3 Discussion]{LST20}) finite or infinite $\mu$ under condition {\bf (H$^{\prime}$)} are derived in Remark  \ref{explicitcritinv}(iii)(a)-(c). \\
Throughout this article, we display various examples and counterexamples which existing literature may not be able to handle. In particular, in Section \ref{section5exam}, we present examples where there are two distinct infinitesimally invariant measures for $(L, C_0^{\infty}(\R^d))$ which are both infinite and also invariant measures (Example \ref{exam:3.1}). 
In this case, we cannot have recurrence and can obtain an explicit example of a bounded co-excessive function that is not constant. We also present an example where there are two distinct infinitesimally invariant measures for $(L, C_0^{\infty}(\R^d))$ which are both finite (Example \ref{Exam:3.3}), or one is finite and the other is infinite (Example \ref{exam:3.4}). In those cases, the infinitesimally invariant measures are not invariant measures. In Example \ref{dualnotconser}, we present a case where $\mu$ and another measure $\widetilde{\mu}$ are two distinct infinitesimally invariant measures for $(L, C_0^{\infty}(\R^d))$ and $(L, C_0^{\infty}(\R^d))$ is not $L^1(\R^d, \mu)$-unique but $L^1(\R^d, \widetilde{\mu})$-unique, and in Example \ref{appstaeple} we show that there exists a closed extension of $(L, C_0^{\infty}(\R^d))$ on $L^1(\R^d, \mu)$ whose $C_0$-semigroup is not sub-Markovian. Finally, we show that the transition semigroup $(P^{\mu}_t)_{t>0}$ of $\M$ is actually the same as the semigroup $(T(t))_{t>0}$ constructed in \cite{LB07} or  \cite{MPW02} under the 
\lq\lq intersection\rq\rq\ of conditions from \cite{LB07} or  \cite{MPW02} and condition {\bf(H)} here. Hence $(T(t))_{t>0}$ inherits various properties derived in this article and \cite{LST20} for $(T^{\mu}_t)_{t>0}$ and vice versa.\\
Although the existing literature \cite{has}, \cite{MT}, \cite{DPZB} can cover a general state space, possibly infinite dimensional, 
there one can find no results about uniqueness of infinitesimally invariant measures as well as non-existence of finite (infinitesimally) invariant measures, since the strong relation between infinitesimally invariant measures and invariant measures is not explicitly identified in that literature. The results established in \cite{St99}, \cite{BKR2}, \cite{GT2}, \cite{LT18}, \cite{LST20} make it possible for us to study (infinitesimally) invariant measures through the analysis of the second order partial differential operator $L$ and through stochastic calculus for the Hunt process associated to $(T^{\mu}_t)_{t>0}$. Since we deal with finite and infinite invariant measures, we can establish the connection to $L^1$-uniqueness and to the existence of operator cores, and since we present various examples and counterexamples for existence and uniqueness of (infinitesimally) invariant measures beyond the case of recurrence, our results can be distinguished from the existing literature. 
We expect that the methods of this article can also be applied in more general or different contexts, for instance to the case of $L$ with more relaxed conditions than those in {\bf (H)} (see \cite{LT19}, \cite{LST20}), the degenerate case (\cite{LT19de}) and the reflection case (\cite{ShTr13a}).

\section{Framework}\label{section1}
The basic notations and conventions that we are going to use and define are those as given in \cite{LT19} up to the following slight modification. In contrast to \cite{LT19} we will use \lq\lq $\mu$\rq\rq\ as a superscript to describe the dependence on the infinitesimally invariant measure for several objects. Thus the notation used here, will only differ in the superscript \lq\lq$\mu$\rq\rq\ used to define some mathematical objects. \\
Unless otherwise stated, we will throughout this article assume the following basic condition on the coefficients $A=(a_{ij})_{1\le i,j\le d}$ and $\mathbf{G} =(g_1, \dots, g_d)$: \\ \\
{\it {\bf (H):} for fixed $p\in (d, \infty)$, $d \geq 2$, $\mathbf{G} = (g_1, \dots, g_d) \in L_{loc}^p(\R^d, \R^d)$,  $A=(a_{ij})_{1\le i,j\le d}$ with $a_{ji}=a_{ij}\in H^{1,p}_{loc}(\R^d) \cap C(\R^d)$ for all $1 \leq i,j \leq d$ is a symmetric matrix of functions
that is locally uniformly strictly elliptic, i.e. for every open ball $B \subset \R^d$ there exist real numbers $\lambda_B, \Lambda_B>0$, such that 
\begin{eqnarray}\label{use}
\lambda_B \left \| \xi \right \|^2 \leq \big \langle A(x) \xi, \xi \big \rangle \ \leq \Lambda_B \left \| \xi \right \|^2  \text{ for all }\; \xi \in \R^d, \; x\in B.
\end{eqnarray}
} 
\centerline{}
\bigskip
For $A$ and $\mathbf{G}$ as in {\bf (H)}, consider the partial differential operator $(L, C_0^{\infty}(\R^d))$
$$
Lf := \frac12 \text{trace}(A \nabla^2 f) + \langle \mathbf{G}, \nabla f\rangle = \frac12 \sum_{i,j=1}^d a_{ij} \partial_{ij} f +\sum_{i=1}^d g_i \partial_i f, \quad f \in C_0^{\infty}(\R^d).
$$
Of course, $L$ can also be applied to $f \in C^2(U)$ where $U \subset \R^d$ is an arbitrary open set with the same expression as above. By \cite[Theorem 1]{BRS}, there exists $\rho \in H^{1,p}_{loc}(\R^d) \cap C(\R^d)$ satisfying $\rho(x)>0$ for all $x \in \R^d$ such that with 
$$
\mu:=\rho dx
$$
it holds
\begin{equation} \label{infinvalfir}
\int_{\R^d} Lf\, d\mu =0, \qquad \forall f \in C_0^{\infty}(\R^d),
\end{equation}
i.e. $\mu$ is an infinitesimally invariant measure for $(L, C_0^{\infty}(\R^d))$ as in the following definition:
\begin{defn} \label{definvaran}
A positive, locally finite measure $\widehat{\mu}$ defined on $(\R^d, \mathcal{B}(\R^d))$ satisfying $Lf \in L^1(\R^d, \widehat{\mu})$ for all $f \in C_0^{\infty}(\R^d)$
is said to be an infinitesimally invariant measure for $(L,C_0^{\infty}(\R^d))$, if
$$
\int_{\R^d} Lf d\widehat{\mu} = 0, \qquad \forall f \in C_0^{\infty}(\R^d).
$$
\end{defn}
We set
$$
\beta^{\rho, A}:=\frac{1}{2} \nabla A + \frac{1}{2\rho}A \nabla \rho,\quad \text{ where }\ 
\nabla A:=\big (\sum_{j=1}^d\partial_j a_{1j},..., \sum_{j=1}^d\partial_j a_{dj}\big ).
$$
Let $(\mathcal{E}^{0, \mu}, D(\mathcal{E}^{0, \mu}))$ be the symmetric Dirichlet form given as the closure of
\begin{equation}\label{symDFdef}
\mathcal{E}^{0, \mu}(f,g):=  \frac12 \int_{\R^d} \langle A \nabla f, \nabla g \rangle d\mu, \quad f,g \in C_0^{\infty}(\R^d)
\end{equation}
on $L^2(\R^d, \mu)$ and let $\mathcal{E}^{0, \mu}_{\alpha}(f,g) := \mathcal{E}^{0, \mu}(f,g)+\alpha\int_{\R^d} fg d\mu$, $\alpha>0$. Denote by $(L^{0, \mu}, D(L^{0, \mu}))$  the associated generator, by $(T^{0, \mu}_t)_{t>0}$ the associated sub-Markovian $C_0$-semigroup of contractions on $L^2(\R^d, \mu)$, and by $(G^{0, \mu}_{\alpha})_{\alpha>0}$ the associated sub-Markovian $C_0$-resolvent of contractions on $L^2(\R^d, \mu)$ (i.e. $G^{0, \mu}_{\alpha}=(\alpha-L^{0, \mu})^{-1}$)  (see \cite[Chapter 1]{FOT} or \cite[Chapter I]{MR}). One easily checks $C_0^{\infty}(\R^d) \subset D(L^{0, \mu})$, 
$$
L^{0, \mu} f= \frac12 \text{trace}(A \nabla^2 f)+\langle \beta^{\rho, A}, \nabla f  \rangle, \quad \forall f \in C_0^{\infty}(\R^d)
$$
and
\begin{equation} \label{infinvsyme}
\int_{\R^d} L^{0, \mu} f d\mu =0, \quad \; \forall f \in C_0^{\infty}(\R^d).
\end{equation}
In particular, 
$$
Lf = L^{0, \mu} f +\langle \mathbf{G}-\beta^{\rho, A}, \nabla f \rangle, \quad f \in C_0^{\infty}(\R^d).
$$
By \cite[Theorem 1.5]{St99}, there exists a closed extension $(\overline{L}^{\mu}, D(\overline{L}^{\mu}))$ of $(L, C_0^{\infty}(\R^d))$ (more precisely, a closed extension of $(L, D(L^{0, \mu})_{0,b})$ on $L^1(\R^d, \mu)$) which generates a sub-Markovian $C_0$-semigroup of contractions $(\overline{T}^{\mu}_t)_{t>0}$ on $L^1(\R^d, \mu)$. Let $(\overline{G}^{\mu}_{\alpha})_{\alpha>0}$ be the sub-Markovian $C_0$-resolvent of contractions on $L^1(\R^d, \mu)$ associated with $(\overline{T}^{\mu}_t)_{t>0}$, i.e. $\overline{G}^{\mu}_{\alpha}:=(\alpha-\overline{L}^{\mu})^{-1}$. Define the partial differential operator $(L'^{, \mu}, C_0^{\infty}(\R^d))$ by
$$
L'^{, \mu} f := \left(\frac{1}{2} \text{trace}(A \nabla^2 f)+\langle \beta^{\rho, A}, \nabla f \rangle \right)-\langle \mathbf{G}- \beta^{\rho, A}, \nabla f   \rangle, \quad f \in C_0^{\infty}(\R^d),
$$
Note that by \eqref{infinvalfir} and \eqref{infinvsyme}, it follows
$$
\int_{\R^d} L'^{, \mu}f\, d\mu =0, \qquad \forall f \in C_0^{\infty}(\R^d),
$$
hence using \cite[Theorem 1.5]{St99} again, there exists a closed extension $(\overline{L}'^{, \mu}, D(\overline{L}'^{, \mu}))$ of $(L'^{, \mu}, C_0^{\infty}(\R^d))$ on $L^1(\R^d, \mu)$ which generates a sub-Markovian $C_0$-semigroup of contractions $(\overline{T}'^{, \mu}_t)_{t>0}$. Let $(\overline{G}'^{, \mu}_{\alpha})_{\alpha>0}$ be the sub-Markovian $C_0$-resolvent associated with $(\overline{T}'^{, \mu}_t)_{t>0}$, i.e. $\overline{G}'^{, \mu}_{\alpha}:=(\alpha-\overline{L}'^{, \mu})^{-1}$. Consider  the restrictions of $(\overline{G}^{\mu}_{\alpha})_{\alpha>0}$, $(\overline{G}'^{, \mu}_{\alpha})_{\alpha>0}$ and $(\overline{T}^{\mu}_t)_{t>0}$, $(\overline{T}'^{, \mu}_t)_{t>0}$ on $L^1(\R^d, \mu)_b$. By Riesz-Thorin interpolation, these can be uniquely extended to sub-Markovian $C_0$-resolvents and to sub-Markovian $C_0$-semigroups of contractions on each $L^r(\R^d, \mu)$, $r\in [1,\infty)$, and to sub-Markovian resolvents and semigroups of contractions on $L^{\infty}(\R^d,\mu)$. We denote these by $(G^{\mu}_{\alpha})_{\alpha>0}$, $(G'^{, \mu}_{\alpha})_{\alpha>0}$ and $(T^{\mu}_t)_{t>0}$, $(T'^{, \mu}_t)_{t>0}$, respectively, independently of the $L^{r}(\R^d, \mu)$-space, $r \in [1, \infty]$ on which they are acting (cf. \cite{GT2}) and we will use these notations from now on exclusively. Thus the notations $(\overline{G}^{\mu}_{\alpha})_{\alpha>0}$, $(\overline{G}'^{, \mu}_{\alpha})_{\alpha>0}$ and $(\overline{T}^{\mu}_t)_{t>0}$, $(\overline{T}'^{, \mu}_t)_{t>0}$ for the resolvents and semigroups acting on $L^1(\R^d, \mu)$ 
will no longer be used and we assume that (unless otherwise stated) the space or spaces on which $(G^{\mu}_{\alpha})_{\alpha>0}$, $(G'^{, \mu}_{\alpha})_{\alpha>0}$ and $(T^{\mu}_t)_{t>0}$, $(T'^{, \mu}_t)_{t>0}$ are acting will be clear from the context.
Using \cite[Remark 1.7 (ii)]{St99}, $(G'^{, \mu}_{\alpha})_{\alpha>0}$ and $(T'^{, \mu}_t)_{t>0}$ are the adjoints of 
$(G^{\mu})_{\alpha>0}$ and $(T^{\mu}_t)_{t>0}$ on $L^2(\R^d, \mu)$, respectively. \medskip \\
We obtained the existence of a density $\rho$ satisfying \eqref{infinvalfir} by \cite[Theorem 1]{BRS} under the assumption of {\bf (H)}. On the other hand, we can just start from a chosen $\rho \in H^{1,p}_{loc}(\R^d) \cap C(\R^d)$ such that $\rho(x)>0$ for all $x \in \R^d$ and such that \eqref{infinvalfir}  holds. Then $(\mathcal{E}^{0, \mu}, D(\mathcal{E}^{0, \mu}))$, $(L^{0, \mu}, D(L^{0, \mu}))$, $(T^{0,\mu}_t)_{t>0}$, $(\overline{L}^{\mu}, D(\overline{L}^{\mu}))$, $(\overline{G}^{\mu}_{\alpha})_{\alpha>0}$, $(\overline{T}^{\mu}_t)_{t>0}$, $(G^{\mu}_{\alpha})_{\alpha>0}$, $(T^{\mu}_{t})_{t>0}$, $(L'^{, \mu}, C_0^{\infty}(\R^d))$, $(\overline{L}'^{, \mu}, D(\overline{L}'^{, \mu}))$, $(\overline{G}'^{, \mu}_{\alpha})_{\alpha>0}$, $(\overline{T}'^{,\mu}_t)_{t>0}$, $(G'^{, \mu}_{\alpha})_{\alpha>0}$, $(T'^{,\mu}_t)_{t>0}$ are defined as in the above way. \\
In that spirit, we can consider the following condition that enables us to deal with a partial differential operator $(L, C_0^{\infty}(\R^d))$ sayisfying {\bf (H)}, where an explicitly choosen measure $\mu$ is an infinitesimally invariant measure for $(L, C_0^{\infty}(\R^d))$ (cf. \cite[Remark 4.5]{LT18}). \\
\text{}\\
{\it
{\bf (H$^{\prime}$)}: for some $p\in (d,\infty)$, $d\ge 2$, $\rho \in H^{1,p}_{loc}(\R^d) \cap C(\R^d)$ with $\rho(x)>0$ for all $x \in \R^d$ is given and $\mu:=\rho dx$. $A=(a_{ij})_{1 \leq i,j \leq d}$ is as in {\bf (H)}.
$\mathbf{G}:=\beta^{\rho, A} + \mathbf{B}$, where $\mathbf{B}\in L^{p}_{loc}(\R^d, \R^d)$ is any vector field satisfying}
$$
\int_{\R^d} \langle\mathbf{B}, \nabla f \rangle d\mu = 0, \quad \;\forall f \in C_0^{\infty}(\R^d).
$$
Under assumption {\bf (H$^{\prime}$)}, it follows that 
$$
Lf = \frac{1}{2}\mathrm{trace}(A \nabla^2 f) +\langle \beta^{\rho, A}+\mathbf{B}, \nabla f \rangle, \quad f \in C_0^{\infty}(\R^d),
$$
satisfies \eqref{infinvalfir} and that {\bf (H)} holds with $A$ and $\mathbf{G}$ defined as in  {\bf (H$^{\prime}$)}. Thus {\bf (H$^{\prime}$)} implies {\bf (H)}, and under either assumption {\bf (H)} or {\bf (H$^{\prime}$)} we have  $\mu=\rho dx$ with $\rho \in H^{1,p}_{loc}(\R^d) \cap C(\R^d)$, and $\rho(x)>0$ for all $x \in \R^d$, but in contrast to assumption 
{\bf (H)} we do not need an existence result (\cite[Theorem 1]{BRS}) to obtain the density $\rho$ under assumption {\bf (H$^{\prime}$)}, because then it is automatically given.
\begin{rem}\label{infini and H}
{\rm The existence of an infinitesimally invariant measure $\mu$ for $(L, C_0^{\infty}(\R^d))$ with nice density $\rho$ as in  \eqref{infinvalfir} follows more generally under condition {\bf (a)} of \cite[Theorem 3.6]{LT19} (see also \cite[Theorem 2.24]{LST20}) which is weaker than condition {\bf (H)} (see \cite[Remark 2.23]{LST20}), but is still covered by the conditions needed for \cite[Theorem 1.5]{St99} which imply the existence of $(\overline{T}^{\mu}_t)_{t>0}$ and $(T^{\mu}_t)_{t>0}$. In the following however, we exclusively concentrate on condition {\bf (H)} (or {\bf (H$^{\prime}$)} which implies {\bf (H)}) because of Remark \ref{equivconinv}(i) below, for which a similar statement on basis of condition {\bf (a)} of \cite[Theorem 3.6]{LT19} is yet not developed.}
\end{rem}

\section{Existence and uniqueness of (infinitesimally) invariant measures in case of recurrence} \label{invrec}
Throughout this section, we assume that $(L, C_0^{\infty}(\R^d))$ satisfies {\bf (H)} of Section \ref{section1} and consider the infinitesimally invariant measure $\mu=\rho dx$ for $(L, C_0^{\infty}(\R^d))$ as specified in Section \ref{section1}.

\begin{defn} \label{definvaran2}
A positive, locally finite measure $\widetilde{\mu}$ defined on $(\R^d, \mathcal{B}(\R^d))$ with $\widetilde{\mu} \ll  \mu$, and $Lf \in L^1(\R^d, \widetilde{\mu})$ for all $f \in C_0^{\infty}(\R^d)$, is called an invariant measure for $(T^{\mu}_t)_{t>0}$ (acting on $L^{\infty}(\R^d, \mu)$), if
$$
\int_{\R^d} T^{\mu}_t 1_A d \widetilde{\mu}  = \widetilde{\mu}(A), \qquad \forall A \in \mathcal{B}(\R^d),\; t>0.
$$
If in the above equation \lq\lq $=$\rq\rq\  is replaced by \lq\lq $\leq$\rq\rq, then $\widetilde{\mu}$ is called a sub-invariant measure for $(T^{\mu}_t)_{t>0}$. $(T^{\mu}_t)_{t>0}$ is called conservative, if
$$
T^{\mu}_t 1_{\mathbb{R}^d} = 1 \; \text{ $\mu$-a.e.\;\; for one (and hence all) $t>0$. }
$$
$(T^{\mu}_t)_{t>0}$ is called non-conservative, if $(T^{\mu}_t)_{t>0}$ is not conservative.
\end{defn}
It will turn out in Lemma \ref{extc0semi}, that if $\widetilde{\mu}$ is an invariant measure for $(T^{\mu}_t)_{t>0}$, then $\widetilde{\mu}$ is an infinitesimally invariant measure for $(L, C_0^{\infty}(\R^d))$. Moreover, an infinitesimally invariant measure is in general only sub-invariant (see \cite[Remark 3.44]{LST20}). It follows in particular from the following Remark \ref{equivconinv}(i) that infinitesimally invariant measures in our setting are always equivalent to Lebesgue measure and have nice densities.  
\begin{rem} \label{equivconinv} \rm
\begin{itemize}
\item[(i)]
If $\widehat{\mu}$ is an infinitesimally invariant measure for $(L, C_0^{\infty}(\R^d))$, then by \cite[Corollary 2.10, 2.11]{BKR2}, there exists $\widehat{\rho} \in H^{1,p}_{loc}(\R^d) \cap C(\R^d)$ such that $\widehat{\rho}(x)>0$ for all $x \in \R^d$ and $\widehat{\mu}= \widehat{\rho} dx$. 
\item[(ii)]
By approximation, one can easily check that
$\mu$ is $(T^{\mu}_t)_{t>0}$-invariant. i.e.
$$
 \int_{\R^d} T^{\mu}_t f d\mu = \int_{\R^d}f d\mu, \qquad \forall f \in L^1(\R^d, \mu), \;t>0,
$$
if and only if
$\mu$ is an invariant measure for $(T^{\mu}_t)_{t>0}$.
\end{itemize}
\end{rem}
The following lemma can be shown exactly as in \cite[Remark 2.13]{LST20}.
\begin{lem} \label{equicondi32}
\begin{itemize}
\item[(i)]
$\mu$ is an invariant measure for $(T^{\mu}_t)_{t>0}$, if and only if $(T'^{, \mu}_t)_{t>0}$ is conservative.
\item[(ii)]
$\mu$ is an invariant measure for $(T'^{, \mu}_t)_{t>0}$, if and only if $(T^{\mu}_t)_{t>0}$ is conservative.
\item[(iii)]
Assume $\mu$ is finite. Then $\mu$ is an invariant measure for $(T^{\mu}_t)_{t>0}$, if and only if $(T^{\mu}_t)_{t>0}$ is conservative.
\item[(iv)]
Assume $\mu$ is finite. Then $\mu$ is an invariant measure for $(T'^{, \mu}_t)_{t>0}$, if and only if $(T'^{, \mu}_t)_{t>0}$ is conservative.
\end{itemize}
\end{lem}
\begin{rem} \label{equivconinv(iii)} \rm
For $r \in [1,p]$, let $(L_r^{\mu}, D(L_r^{\mu}))$ denote the generator associated with $(T^{\mu}_t)_{t>0}$ on $L^r(\R^d, \mu)$. If $f \in C_0^{\infty}(\R^d)$, then $f \in L^r(\R^d, \mu)$ and $\overline{L}^{\mu} f= Lf \in L^r(\R^d, \mu)$, hence by \cite[Lemma 1.11]{Eb}, it holds $f \in D(L_r^{\mu})$ and $L_r^{\mu} f = \overline{L}^{\mu}f =Lf$. Therefore, $(L_r^{\mu}, D(L_r^{\mu}))$ is a closed extension of  $(L, C_0^{\infty}(\R^d))$, which generates a sub-Markovian $C_0$-semigroup $(T^{\mu}_t)_{t>0}$ on $L^r(\R^d, \mu)$.
Thus by \cite[Theorem 1.2]{Eb}, the following statements {\it (a)} and {\it (b)} are equivalent:
\begin{itemize} \it
\item[(a)]
$C_0^{\infty}(\R^d)$ is an operator core for  $(L_{r}^{\mu}, D(L_{r}^{\mu}))$, i.e. $C_0^{\infty}(\R^d)$ is dense in $D(L_{r}^{\mu})$ with respect to the graph norm
$$
\|f\|_{D(L_{r}^{\mu})}  := \|f\|_{L^r(\R^d, \mu)}+ \|L^{\mu} f \|_{L^1(\R^d, \mu)}, \; \quad f \in D(L_r^{\mu})
$$
In other words,
$(L_{r}^{\mu}, D(L_{r}^{\mu}))$ is the closure of $(L, C_0^{\infty}(\R^d))$ on $L^r(\R^d, \mu)$.
\item[(b)]
$(L, C_0^{\infty}(\R^d))$ is $L^r(\R^d, \mu)$-unique, i.e. $(T^{\mu}_t)_{t>0}$ is the only $C_0$-semigroup on $L^r(\R^d, \mu)$, whose generator extends $(L, C_0^{\infty}(\R^d))$.
\end{itemize}
Since $G_{\alpha}^{\mu} = (\alpha-L^{\mu}_r)^{-1}$ for all $\alpha>0$, {\it (a)} is equivalent to 
\begin{itemize}\it
\item[(c)] $(\alpha-L)(C^{\infty}_0(\R^d)) \subset L^r(\R^d, \mu)$ dense with respect to $\|\cdot \|_{L^r(\R^d, \mu)}$ for any $\alpha>0$.
\end{itemize}
By the Hahn-Banach theorem (\cite[Proposition 1.9]{BRE}) and the Riesz representation theorem (\cite[Theorem 4.11, 4.14]{BRE}), {\it(c)} is equivalent to
\begin{itemize} \it
\item[(d)]
For $\alpha>0$ and $h \in L^{q}(\R^d, \mu)$ with $q \in (1, \infty]$ and $\frac{1}{r}+\frac{1}{q}=1$,
$$
\int_{\R^d} (\alpha-L) u \cdot h d\mu =0, \quad \forall u \in C_0^{\infty}(\R^d)
$$
implies that $h=0$ $\mu$-a.e.
\end{itemize}
In particular, by \cite[Corollary 2.2]{St99} and Remark \ref{equivconinv}(ii), $\mu$ is an invariant measure for $(T^{\mu}_t)_{t>0}$, if and only if one of the above equivalent conditions {\it (a)-(d)} hold with $r=1$.
\end{rem}
\text{}\\
Let $G^{\mu}$, $G'^{, \mu}$ be the potential operators of $(T^{\mu}_t)_{t>0}$ and $(T'^{, \mu}_t)_{t>0}$ respectively, defined for $f \in L^1(\R^d, \mu)$, $f  \geq 0$ $\mu$-a.e., through monotone convergence by
$$
G^{\mu}f:= \lim_{N \rightarrow \infty} \int_0^N T^{\mu}_t f \, d\mu = \lim_{\alpha \rightarrow 0+ } \int_{0}^{\infty} e^{-\alpha t} T^{\mu}_t f \, d\mu =  \lim_{\alpha \rightarrow 0+} G^{\mu}_{\alpha} f, \quad \mu\text{-a.e.}
$$
$$
G'^{, \mu}f:= \lim_{N \rightarrow \infty} \int_0^N T'^{, \mu}_t f \, d\mu = \lim_{\alpha \rightarrow 0+ } \int_{0}^{\infty} e^{-\alpha t} T'^{, \mu}_t f \, d\mu=  \lim_{\alpha \rightarrow 0+} G'^{, \mu}_{\alpha} f, \quad \mu\text{-a.e.}
$$

\begin{defn} \label{defrectra}
$(T^{\mu}_t)_{t>0}$ is said to be recurrent, if for any $f \in L^1(\R^d, \mu)$ with $f \geq 0$ $\mu$-a.e. we have
$$
G^{\mu}f  \in \{0, \infty\}, \;\; \mu \text{-a.e.},
$$
i.e. $\mu(\{0< G^{\mu}f < \infty\})=0$.\\
$(T^{\mu}_t)_{t>0}$ is said to be transient, if there exists $g \in L^1(\R^d, \mu)$ with $g> 0$  $\mu$-a.e. such that
$$
G^{\mu}g <\infty, \;\; \mu \text{-a.e.}
$$
Let $(S_t)_{t>0}$ be a sub-Markovian $C_0$-semigroup of contractions  on $L^r(\R^d, \mu)$ for some $r \in [1, \infty)$ that satisfies $(S_t)_{t>0} = (T^{\mu}_t)_{t>0}$ on $L^r(\R^d, \mu)$. Then $(S_t)_{t>0}$ is called recurrent (resp. transient), if $(T^{\mu}_t)_{t>0}$ is recurrent (resp. transient).
\end{defn} 
Applying \cite[Proposition 4.9]{LT18} (or \cite[Theorem 3.38(i)]{LST20}) to $(T^{\mu}_t)_{t>0}$ and  $(T'^{, \mu}_t)_{t>0}$, we have:
\begin{prop} \label{strictirreducibilityprop}
$(T^{\mu}_t)_{t>0}$ is either recurrent or transient and $(T'^{, \mu}_t)_{t>0}$ is either recurrent or transient.
\end{prop}
The statement of the following theorem is already known from  \cite[Proposition 2.5]{BCR} which is derived for a sub-Markovian resolvent of kernels with respect to a $\sigma$-finite sub-invariant measure in a context of weak duality. It was also previously derived in the framework of lower bounded semi-Dirichlet forms (cf. \cite[Theorem 1.3.4]{Osh}). We present here an alternative proof, which can also be used in more general situations than the one considered here.
\begin{theo}\label{cosemirec}
If $(T^{\mu}_t)_{t>0}$ is recurrent, then $(T'^{, \mu}_t)_{t>0}$ is recurrent.
\end{theo}
\begin{proof}
{\bf Step 1)}
Let $f \in L^1(\R^d, \mu)$ be such that $f \geq 0$  $\mu$-a.e. and let $Z:= \{x \in \R^d : G^{\mu} f(x)=0 \}$. Then $0 \leq 1_Z G^{\mu}_{\alpha} f \leq 1_Z G^{\mu} f = 0$ $\mu$-a.e. for all $\alpha>0$. By the $L^1(\R^d, \mu)$-strong continuity of $(G^{\mu}_{\alpha})_{\alpha>0}$,
\begin{equation*} 
\int_{\R^d} 1_Z f \,d\mu = \lim_{\alpha \rightarrow \infty} \int_{\R^d} 1_Z \alpha G_{\alpha}^{\mu}f\, d\mu= 0,
\end{equation*}
hence $\mu ( \{f>0\} \cap Z) = 0$ and therefore $G^{\mu}f = \infty$ $\mu$-a.e. on $\{ f>0 \}$. \\
{\bf Step 2) } Now let $g \in L^1(\R^d, \mu)$ with $g> 0$ $\mu$-a.e. We will show $\mu ( \{ 0<G'^{, \mu} g< \infty \}) =0$.\\
First, suppose there exists $\varepsilon \in (0, \infty)$ such that $\mu(0< G'^{, \mu} g < \infty)= \varepsilon$. Then there exists $N \in \N$ such that $\frac{\varepsilon}{2}\leq \mu( 0< G'^{, \mu} g <N) \leq \varepsilon$. Define $U:= \{ 0<G'^{, \mu} g<N \}$. Then as we showed at Step 1, it holds $G^{\mu} 1_U = \infty$ $\mu$-a.e. on $U$ and by the monotone convergence theorem,
\begin{eqnarray*}
\infty  &=& \int_{\R^d} G^{\mu} 1_U \cdot g \, d\mu  =  \lim_{\alpha \rightarrow 0+} \int_{\R^d} G^{\mu}_{\alpha} 1_U \cdot g \, d\mu  \\
&=& \lim_{\alpha \rightarrow 0+}  \int_{\R^d} 1_U \cdot G'^{, \mu}_{\alpha} g d\mu = \int_{\R^d} 1_U G'^{, \mu} g \, d\mu \leq N \varepsilon,
\end{eqnarray*}
which is contradiction. Thus $\mu(0< G'^{, \mu} g < \infty) \in \{0, \infty \}$.
 \\
Next, suppose $\mu(0< G'^{, \mu} g < \infty)= \infty$. Then there exists $R>0$ such that 
$$
0< \mu \big (  \{ 0< G'^{, \mu} g < \infty  \} \cap B_R   \big ) \leq \mu(B_R)< \infty,
$$
and so there exists $M \in \N$ such that
$$
\frac{\delta}{2}\leq \mu \big(  \{ 0< G'^{, \mu} g < M  \} \cap B_R   \big) \leq  \delta:= \mu \big(  \{ 0< G'^{, \mu} g< \infty  \} \cap B_R   \big) \in (0, \infty).
$$
Define $V:= \{ 0< G'^{, \mu} g < M  \} \cap B_R$. Since $G^{\mu} 1_V= \infty$ $\mu$-a.e. on $V$ from Step 1, we have by the monotone convergence theorem,
\begin{eqnarray*}
\infty  &=& \int_{\R^d} G^{\mu} 1_V \cdot g \, d\mu  =  \lim_{\alpha \rightarrow 0+}  \int_{\R^d} G^{\mu}_{\alpha} 1_V \cdot g \, d\mu \\
&=& \lim_{\alpha \rightarrow 0+} \int_{\R^d} 1_U \cdot G'^{, \mu}_{\alpha} g d\mu  = \int_{\R^d} 1_V G'^{, \mu} g \, d\mu \leq M \delta,
\end{eqnarray*}
which is contradiction. Therefore, $\mu ( \{ 0<G'^{, \mu} g< \infty \}) =0$. \\
{\bf Step 3)} Suppose that $(T'^{, \mu}_t)_{t>0}$ is transient. Then by Step 2, $G'^{, \mu} g =0$ $\mu$-a.e. for some $g \in L^1(\R^d, \mu)$ with $g> 0$ $\mu$-a.e. Therefore, $0 \leq G'^{, \mu}_{\alpha}g \leq G'^{, \mu} g =0$ for any $\alpha>0$, so that $g = \lim_{\alpha \rightarrow \infty} \alpha G'^{,\mu}_{\alpha}g =0$ in $L^1(\R^d, \mu)$, which contradicts to the assumption of $g>0$ $\mu$-a.e. Therefore, $(T'^{, \mu}_t)_{t>0}$ is not transient, and $(T'^{, \mu}_t)_{t>0}$ is recurrent by Proposition \ref{strictirreducibilityprop}.
\end{proof}

\begin{cor} \label{existinvariancemea}
If $(T^{\mu}_t)_{t>0}$ is recurrent, then both $(T^{\mu}_t)_{t>0}$ and $(T'^{, \mu}_t)_{t>0}$ are conservative. In particular, $\mu$ is an invariant measure for both $(T^{\mu}_t)_{t>0}$ and $(T'^{, \mu}_t)_{t>0}$.
\end{cor}
\begin{proof}
By Theorem \ref{cosemirec}, $(T'^{, \mu}_t)_{t>0}$ is recurrent. Thus, $(T^{\mu}_t)_{t>0}$ and $(T'^{,  \mu}_t)_{t>0}$ are both conservative by \cite[Corollary 20]{GT2} and the last assertion follows from Lemma \ref{equicondi32}(i), (ii). 
\end{proof}
\text{}\\
Using the framework of \cite[Section 1]{St99}, we will present a new approach to obtain the uniqueness of infinitesimally invariant measures for $(L, C_0^{\infty}(\R^d))$ and invariant measures for $(T^{\mu}_t)_{t>0}$. For  $i\in \{1,2\}$, let $\mu_i=\rho_i dx$ be infinitesimally invariant measures for $(L, C_0^{\infty}(\R^d))$ such that $\rho_i \in H^{1,p}_{loc}(\R^d) \cap C(\R^d)$ and $\rho_i(x)>0$ for all $x \in \R^d$ (see Remark \ref{equivconinv}(i)). Let $V$ be a bounded open subset of $\R^d$ and $(\mathcal{E}^{0,V, \mu_i}, D(\mathcal{E}^{0,V, \mu_i}))$ be the symmetric Dirichlet form defined as the closure of
$$
\mathcal{E}^{0,V, \mu_i}(u,v) :=\frac12 \int_{V} \langle A \nabla u, \nabla v \rangle d\mu_i, \quad u,v \in C_0^{\infty}(V)
$$
on $L^2(V, \mu_i)$. Denote the associated generator by $(L^{0, V, \mu_i}, D(L^{0, V, \mu_i}))$  and the corresponding sub-Markovian $C_0$-semigroup of contractions on $L^2(V, \mu_i)$ by $(T^{0,V, \mu_i}_t)_{t>0}$. Since $(T^{0,V, \mu_i}_t)_{t>0}$ is symmetric and sub-Markovian, $(T^{0,V, \mu_i}_t)_{t>0}|_{L^2(V, \mu_i)_b}$ can be extended to a sub-Markovian $C_0$-semigroup of contractions $(\overline{T}^{0,V, \mu_i}_t)_{t>0}$ on $L^1(V, \mu_i)$. Let $(\overline{L}^{0, V, \mu_i}, D(\overline{L}^{0, V, \mu_i}))$ be the generator on $L^1(V, \mu_i)$ associated with $(\overline{T}^{0,V, \mu_i}_t)_{t>0}$. Let $(\overline{L}^{V, \mu_i}, D(\overline{L}^{V, \mu_i}))$ be the closure of 
$$
\Big(L^{0,V, \mu_i}+\langle \mathbf{G}-\beta^{\rho_i, A}, \nabla \rangle,  D(L^{0,V, \mu_i})_b\Big)
$$ 
on $L^1(V, \mu_i)$, as constructed in \cite[Proposition 1.1(i)]{St99}. Then we obtain the following results.
\begin{lem} \label{intermedlem}
\begin{itemize}
\item[(i)] 
$D(\overline{L}^{0, V, \mu_1})_b = D(\overline{L}^{0,V, \mu_2})_b$ and 
\begin{equation*} \label{substracteq}
\overline{L}^{0, V, \mu_2} u  =\overline{L}^{0, V, \mu_1} u  + \langle \frac{A \nabla \rho_2}{2 \rho_2}- \frac{A \nabla \rho_1}{2 \rho_1}, \nabla u  \rangle, \quad \forall u \in D(\overline{L}^{0, V, \mu_1})_b.
\end{equation*}
\item[(ii)]
$D(\overline{L}^{V, \mu_1})=D(\overline{L}^{V, \mu_2})$ and $\overline{L}^{V, \mu_1} u = \overline{L}^{V, \mu_2}u$,\quad $\forall u \in D(\overline{L}^{V, \mu_1})$.

\end{itemize}
\end{lem}
\begin{proof}
(i) Since $\rho_i$ with $i \in \{1,2\}$ is locally bounded below and above by strictly positive constants,  it holds $D(\mathcal{E}^{0,V, \mu_1})= D(\mathcal{E}^{0,V, \mu_2})$. Let $\psi=\frac{\rho_2}{\rho_1}$. By symmetry, it is enough to show $D(\overline{L}^{0,V, \mu_1})_b \subset D(\overline{L}^{0,V, \mu_2})_b$. Let $u \in D(\overline{L}^{0,V, \mu_1})_b$ and $v \in C_0^{\infty}(V)$. Then $u \in D(\mathcal{E}^{0,V, \mu_1})$  by \cite[Lemma 1.2(i)]{St99}, and $v \psi \in D(\mathcal{E}^{0, V, \mu_1})$. 
By \cite[Lemma 1.2(iii)]{St99},
$$
-\int_{V} \overline{L}^{0,V, \mu_1} u\cdot v d\mu_2= -\int_{V} \overline{L}^{0,V, \mu_1} u\cdot v\psi  d\mu_1 = \mathcal{E}^{0,V, \mu_1}(u, v\psi),
$$
and
\begin{eqnarray*}
\mathcal{E}^{0,V, \mu_1}(u,v \psi) = \frac{1}{2} \int_V \langle A\nabla u, \nabla(v \psi)  \rangle d\mu_1  = \frac12 \int_{V} \langle A \nabla u, \nabla v   \rangle \psi \rho_1 dx +\frac12 \int_{V}  \langle  A \nabla u, \nabla \psi  \rangle v d\mu_1.
\end{eqnarray*}
Therefore,       
$$
\mathcal{E}^{0,V, \mu_2}(u,v)= -\int_{V} \overline{L}^{0,V, \mu_1} u\cdot v  d\mu_2 - \frac{1}{2} \int_V \langle A \nabla u, \frac{\nabla \psi}{\psi} \rangle v d\mu_2,
$$
and  $\overline{L}^{0,V, \mu_1} u$ $+ \frac12 \langle A \nabla u, \frac{\nabla \psi}{\psi} \rangle \in L^1(V, \mu_2)$. Hence by \cite[I. Lemma 4.2.2.1]{BH}, $u \in D(\overline{L}^{0, V, \mu_2})_b$ and 
$$
\overline{L}^{0, V, \mu_2} u =\overline{L}^{0, V, \mu_1} u  +\frac12 \langle A \nabla u, \frac{\nabla \psi}{\psi} \rangle = \overline{L}^{0, V, \mu_1} u  + \langle \frac{A \nabla \rho_2}{2 \rho_2}- \frac{A \nabla \rho_1}{2 \rho_1}, \nabla u  \rangle.
$$ 
(ii) By the step 1 in the proof of \cite[Proposition 1.1 (ii)]{St99}, $D(\overline{L}^{0,V, \mu_i})_b \subset D(\overline{L}^{V, \mu_i})$ and
$$
\overline{L}^{V, \mu_i}u=\overline{L}^{0,V, \mu_i} u + \langle \mathbf{G}-\beta^{\rho_i,A}, \nabla u \rangle, \quad u \in D(\overline{L}^{0,V, \mu_i})_b.
$$
Thus for any $u \in D(\overline{L}^{0,V, \mu_1})_b =D(\overline{L}^{0,V, \mu_2})_b$, it holds by (i)
\begin{eqnarray*}
\overline{L}^{V, \mu_2} u &=&  \overline{L}^{0,V, \mu_2} u + \langle \mathbf{G}- \beta^{\rho_2, A}, \nabla u \rangle \\
&=& \overline{L}^{0, V, \mu_1} u  + \langle \frac{A \nabla \rho_2}{2 \rho_2}- \frac{A \nabla \rho_1}{2 \rho_1}, \nabla u  \rangle+ \langle \mathbf{G}- \beta^{\rho_2, A}, \nabla u \rangle \\
&=& \overline{L}^{0,V, \mu_1} u + \langle \mathbf{G}- \beta^{\rho_1, A}, \nabla u \rangle = \overline{L}^{V, \mu_1} u.
\end{eqnarray*}
Observe that $D(L^{0, V, \mu_i}) \subset D(\overline{L}^{0, V, \mu_i})$ with $\overline{L}^{0, V, \mu_i} v =L^{0, V, \mu_i} v$  on $L^1(V, \mu_i)$ for all $v \in D(L^{0, V, \mu_i})$ by \cite[Lemma 1.12]{Eb}.
Since $(\overline{L}^{V, \mu_i}, D(\overline{L}^{V, \mu_i}))$ is the closure of $\big(L^{0, V, \mu_i}+\langle \mathbf{G}-\beta^{\rho_i, A}, \nabla \rangle, D(L^{0, V, \mu_i})_b\big)$ on $L^1(V, \mu_i)$ by \cite[Proposition 1.1]{St99}, the assertion follows.
\end{proof}

\begin{theo}\label{pathuniquenesshunt}
Let $\mu_i=\rho_i dx$, $i\in \{1,2\}$, be infinitesimally invariant measures for $(L, C_0^{\infty}(\R^d))$ where the $\rho_i$ have the regularity of Remark \ref{equivconinv}(i). Then for any $f \in L^1(\R^d, \mu_1) \cap L^1(\R^d, \mu_2)$ and $t>0$,
$$
T^{\mu_1}_t f = T^{\mu_2}_t f, \;\; \text{ a.e.}   \text{ on $\R^d$.}
$$
In particular, $(T^{\mu_1}_t)_{t>0}$ is conservative, if and only if $(T^{\mu_2}_t)_{t>0}$ is conservative and $(T^{\mu_1}_t)_{t>0}$ is recurrent, if and only if $(T^{\mu_2}_t)_{t>0}$ is recurrent. Moreover, $\widetilde{\mu}$ is an invariant measure for $(T^{\mu_1}_t)_{t>0}$, if and only if $\widetilde{\mu}$ is an invariant measure for $(T^{\mu_2}_t)_{t>0}$.
\end{theo}
\begin{proof} 
Let $i \in \{1,2 \}$, and $(\overline{G}_{\alpha}^{V, \mu_i})_{\alpha>0}$ be the resolvents defined by $\overline{G}_{\alpha}^{V, \mu_i} := (\alpha-\overline{L}^{V, \mu_i})^{-1}$. Then by Lemma \ref{intermedlem},
\begin{equation} \label{locresequ}
\overline{G}_{\alpha}^{V, \mu_1} f = \overline{G}_{\alpha}^{V, \mu_2} f, \qquad \forall f \in L^1(V, \mu_1)= L^1(V, \mu_2).
\end{equation}
By \eqref{locresequ} and the construction of $(G^{\mu_i}_{\alpha})_{\alpha>0}$ on $L^1(\R^d, \mu_i)$ in \cite[Theorem 1.5(b)]{St99}, 
$$
G^{\mu_1}_{\alpha} f  = G^{\mu_2}_{\alpha} f \,\text{\; a.e., \, $\forall f \in L^1(\R^d, \mu_1) \cap L^1(\R^d, \mu_2)$, $\alpha>0$}.
$$ 
Thus, for $f \in L^1(\R^d, \mu_1) \cap L^1(\R^d, \mu_2)$, $\alpha>0$, it holds for an arbitrary bounded Borel set $B$
$$
1_B G^{\mu_1}_{\alpha} f  = 1_B G^{\mu_2}_{\alpha} f \;\; \text{ in \;$L^1(B)$},
$$
hence
\begin{equation} \label{laplacetrans} 
\int_0^{\infty}  e^{-\alpha t} 1_BT^{\mu_1}_t f dt =\int_0^{\infty} e^{-\alpha t} 1_B T^{\mu_2}_t f dt \;\; \;\text{ in $L^1(B)$}.
\end{equation}
As a consequence of the Post-Widder inversion formula (\cite[Theorem 1.7.3]{ACHN}) and
the right continuity of $t \in (0, \infty) \mapsto 1_B T^{\mu_i}_t f$ in $L^1(B)$ for $i \in \{1, 2\}$, \eqref{laplacetrans} implies
$$
1_B T^{\mu_1}_t f =1_B T^{\mu_2}_t f, \quad \text{ in $L^1(B)$}, \;\; \forall t>0.
$$
Therefore, the first assertion follows and the remaining assertions are easily derived.
\end{proof}

\begin{prop} \label{prop1:9}
Assume {\bf (H$^{\prime}$)} and let $\mathbf{B}\equiv 0$, i.e. $\mathbf{G}=\beta^{\rho,A}$.
Then
$$
T^{\mu}_t f = T^{0, \mu}_t f, \quad \; \forall f \in L^2(\R^d, \mu),\; t>0,
$$
where $(T^{0, \mu}_t)_{t>0}$ is the semigroup associated with the symmetric Dirichlet form defined in \eqref{symDFdef}.
\end{prop}
\begin{proof}
By \cite[Theorem 1.5(c)]{St99}, for $f\in L^1(\R^d, \mu)_b$, $\alpha>0$ it holds $G^{\mu}_{\alpha}f  \in D(\overline{L}^{\mu})_b\subset D(\mathcal{E}^{0, \mu})$, and for any $v\in C_0^{\infty}(\R^d)$ it follows from \cite[Theorem 1.5(c)]{St99} and \cite[Theorem 2.8]{MR} that
$$
\mathcal{E}^{0, \mu}_{\alpha}(G^{\mu}_{\alpha}f,v)=\int_{\R^d}(\alpha-\overline{L}^{\mu}) G^{\mu}_{\alpha} f \cdot vd\mu=\int_{\R^d}f vd\mu=\mathcal{E}^{0, \mu}_{\alpha}(G^{0, \mu}_{\alpha}f,v).
$$
Thus $G^{\mu}_{\alpha}f =G^{0, \mu}_{\alpha} f$. By \cite[Theorem 1.7.3]{ACHN} and the right continuity of $t \in (0, \infty) \mapsto T^{\mu}_t f$ (and $T^{0, \mu}_t f$) in $L^2(\R^d, \mu)$, the assertion follows as in the proof of Theorem \ref{pathuniquenesshunt}.
\end{proof}

\begin{rem} \label{nonsymmetricequi} \rm
Assume {\bf (H$^{\prime}$)} and that there exists a constant $M>0$ such that
$$
\left| \int_{\R^d} \langle \mathbf{B}, \nabla f \rangle g d\mu \right| \leq M \mathcal{E}_1^{0, \mu}(f,f)^{1/2} \mathcal{E}_1^{0, \mu}(g,g)^{1/2}, \quad \forall f,g \in C_0^{\infty}(\R^d),
$$
where $(\mathcal{E}^{0, \mu}, D(\mathcal{E}^{0, \mu}))$ is defined as in \eqref{symDFdef}.
Let $(S_t)_{t>0}$ be the sub-Markovian $C_0$-semigroup of contractions on $L^2(\R^d, \mu)$ associated with the (possibly non-symmetric) weak sectorial Dirichlet form $(\mathcal{B}, D(\mathcal{B}))$ defined as the closure (see \cite[Chapter I]{MR}) of
\begin{equation} \label{sectodirichlet}
\mathcal{B}(f,g):=\frac12 \int_{\R^d}  \langle A \nabla f, \nabla g \rangle d\mu-\int_{\R^d} \langle \mathbf{B}, \nabla f \rangle g d\mu, \quad f,g \in C_0^{\infty}(\R^d).
\end{equation}
Using a similar method as in Proposition \ref{prop1:9}, it holds
$$
T^{\mu}_t f = S_t f, \quad \; \forall f \in L^2(\R^d, \mu),\, t>0.
$$
\end{rem}
\text{}\\
By Proposition \ref{prop1:9} and Remark \ref{nonsymmetricequi}, recurrence and transience of the semigroup associated with a weak sectorial Dirichlet form defined as the closure of \eqref{sectodirichlet} can be investigated in the here considered framework of generalized Dirichlet forms (see Definition \ref{defrectra}). \\
Under assumption {\bf (H)}, the statement of the following proposition can be concluded from \cite[Theorems 3.38(iii) and 3.46]{LST20}, but we emphasize that it also holds under the mere assumption {\bf (a)} of \cite{LST20} or \cite{LT19} (see proof of Proposition \ref{strictirreducibilityprop}), which is more general than the assumptions used in \cite[Theorems 3.38(iii) and 3.46]{LST20}.
\begin{prop} \label{fininvariant}
Assume $\mu$ is finite. If $\mu$ is an invariant measure for $(T^{\mu}_t)_{t>0}$ (or equivalently by Lemma \ref{equicondi32}(iii), $(T^{\mu}_t)_{t>0}$ is conservative), then $(T^{\mu}_t)_{t>0}$ is recurrent.
\end{prop}
\begin{proof}
Since $T^{\mu}_t 1_{\mathbb{R}^d} = 1$ $\mu$-a.e. for all $t>0$, we get $\alpha G^{\mu}_{\alpha} 1_{\mathbb{R}^d} = 1$ $\mu$-a.e. for all $\alpha>0$ and therefore, $G^{\mu} 1_{\mathbb{R}^d} = \infty$ $\mu$-a.e. Since $\mu$ is finite, $1_{\mathbb{R}^d}\in L^1(\R^d,\mu)$. Then $(T^{\mu}_t)_{t>0}$ is not transient by \cite[Remark 3(a)]{GT2}, hence $(T^{\mu}_t)_{t>0}$ is recurrent by Proposition \ref{strictirreducibilityprop}.
\end{proof}

\begin{theo}\label{theo1.12}
Let $\mu=\rho dx$ and $\widetilde{\mu}=\widetilde{\rho} dx$ be infinitesimally invariant measures for $(L, C_0^{\infty}(\R^d))$ such that $\rho$, $\widetilde{\rho} \in H^{1,p}_{loc}(\R^d) \cap C(\R^d)$, $\rho(x)$, $\widetilde{\rho}(x)>0$ for all $x \in \R^d$ and $\frac{\rho}{\widetilde{\rho}} \in \mathcal{B}_b(\R^d)$.
Then it holds
$$
T'^{, \widetilde{\mu}}_t \left(\frac{\rho}{\widetilde{\rho}}\right) \leq \frac{\rho}{\widetilde{\rho}}, \quad \text{$\mu$-a.e.} \; \text{ for all } t>0.
$$
If additionally $\mu$ is an invariant measure for $(T^{\mu}_t)_{t>0}$, then 
$$
T'^{, \widetilde{\mu}}_t \left(\frac{\rho}{\widetilde{\rho}}\right) = \frac{\rho}{\widetilde{\rho}}, \quad \text{$\mu$-a.e.} \; \text{ for all } t>0.
$$

\end{theo}
\begin{proof}
Let $g \in C_0^{\infty}(\R^d)$ and $t>0$. Using Theorem \ref{pathuniquenesshunt},  $T^{\mu}_t g = T^{\widetilde{\mu}}_t g$, a.e. Thus, for $n \in \N$,
\begin{eqnarray*}
\int_{\R^d} T'^{, \mu}_t 1_{B_n} \cdot g d\mu = \int_{\R^d} 1_{B_n} T^{\mu}_t g \cdot \rho dx  = \int_{\R^d}  1_{B_n} \frac{\rho}{\widetilde{\rho}} \cdot T^{\widetilde{\mu}}_t g \,d\widetilde{\mu}  = \int_{\R^d} \frac{\widetilde{\rho}}{\rho} \cdot T'^{, \widetilde{\mu}}_t \left(1_{B_n} \frac{\rho}{\widetilde{\rho}}\right)  \cdot g
d\mu.
\end{eqnarray*}
Consequently, 
$$
T'^{, \widetilde{\mu}}_t\left(1_{B_n} \frac{\rho}{\widetilde{\rho}}\right)=\frac{\rho}{\widetilde{\rho}} \cdot T'^{, \mu}_t 1_{B_n}, \quad \text{$\mu$-a.e.}
$$
$(T'^{, \mu}_t)_{t>0}$ is sub-Markovian and if $\mu$ is an invariant measure for $(T^{\mu}_t)_{t>0}$, then $(T'^{, \mu}_t)_{t>0}$ is conservative by Lemma \ref{equicondi32}(i). Therefore, letting $n \rightarrow \infty$, the assertion follows.
\end{proof}

\begin{theo} \label{preinvlemma}
Assume that $(T^{\mu}_t)_{t>0}$ is recurrent. Then $\mu$ is the unique infinitesimally invariant measure for $(L, C_0^{\infty}(\R^d))$ up to a multiplicative constant, i.e. if $\mu_1$ is an infinitesimally invariant measure for $(L, C_0^{\infty}(\R^d))$, then there exists a constant $c>0$ such that $\mu_1= c \mu$.
\end{theo}
\begin{proof}
If $\mu_1$ is an infinitesimally invariant measure for $(L, C_0^{\infty}(\R^d))$, then by \cite[Corollary 2.10, 2.11]{BKR2}, there exists $\rho_1 \in H^{1,p}_{loc}(\R^d) \cap C(\R^d)$
such that $\rho_1(x)>0$ for all $x \in \R^d$ and $\mu_1= \rho_1 dx$. Let $\widetilde{\rho}:=\rho+\rho_1$, $\widetilde{\mu}:= \widetilde{\rho} dx$. Then by linearity, $\widetilde{\mu}$ is an infinitesimally invariant measure for $(L, C_0^{\infty}(\R^d))$. Since $(T^{\mu}_t)_{t>0}$ is recurrent,  $\mu$ is an invariant measure for $(T^{\mu}_t)_{t>0}$ by Corollary \ref{existinvariancemea}. Thus by Theorem \ref{theo1.12}, 
$$
T'^{, \widetilde{\mu}}_t \left(\frac{\rho}{\widetilde{\rho}}\right) = \frac{\rho}{\widetilde{\rho}}, \quad \text{$\mu$-a.e.} \; \text{ for all } t>0.
$$
Note that $(T^{\widetilde{\mu}}_t)_{t>0}$ is recurrent by Theorem \ref{pathuniquenesshunt}, so that $(T'^{, \widetilde{\mu}}_t)_{t>0}$ is also recurrent by Theorem \ref{cosemirec}.
 By \cite[Proposition 11(b)]{GT2} and \cite[Theorem 3.12, Corollary 4.8(i) and Proposition 3.10(ii)]{LT18}, there exists a constant $\widetilde{c} \in (0,1)$ such that  $\frac{\rho}{\rho+\rho_1} = \widetilde{c}$. Therefore, we obtain $\rho_1 = \frac{1-\widetilde{c}}{\widetilde{c}} \rho$, as desired.
\end{proof}

\begin{lem} \label{extc0semi}
Let $\widetilde{\mu}=\widetilde{\rho} dx$ with $\widetilde{\rho} \in L^1_{loc}(\R^d)$ be an invariant measure for $(T^{\mu}_t)_{t>0}$. Then $(T^{\mu}_t)_{t>0}|_{C_0(\R^d)}$ can  uniquely be extended to a sub-Markovian $C_0$-semigroup of contractions $(\widetilde{T}_t)_{t>0}$ on $L^1(\R^d, \widetilde{\mu})$ and $T^{\mu}_t f  = \widetilde{T}_t f$ $\widetilde{\mu}$-a.e. for any $f \in L^1(\R^d, \mu) \cap L^1(\R^d, \widetilde{\mu})$, $t>0$. If $(\widetilde{L}, D(\widetilde{L}))$ denotes the generator of $(\widetilde{T}_t)_{t>0}$ on $L^1(\R^d, \widetilde{\mu})$, then $C_0^{\infty}(\R^d)\subset D(\widetilde{L})$ and $Lg = \widetilde{L}g$ for all $g \in C_0^{\infty}(\R^d)$. In particular, $\widetilde{\mu}$ is an infinitesimally invariant measure for $(L, C_0^{\infty}(\R^d))$ and $(T^{\widetilde{\mu}}_t)_{t>0} = (\widetilde{T}_t)_{t>0}$.
\end{lem}
\begin{proof}
Let $t>0$. If $h$ is a positive simple function, then by linearity
 $$
\int_{\R^d} T^{\mu}_t h \, d\widetilde{\mu} =  \int_{\R^d} h \, d\widetilde{\mu}.
$$
Let $g \in C_0(\R^d)$. Splitting $g= g^+ - g^-$ and using  monotone approximations through positive simple functions for $g^+$ and $g^-$, we get
\begin{equation} \label{invaproperty}
\int_{\R^d} T^{\mu}_t g  d\widetilde{\mu} = \int_{\R^d} g d\widetilde{\mu}.
\end{equation}
Moreover, since $(T^{\mu}_t)_{t>0}$ is positivity preserving and $\widetilde{\mu} \ll \mu$,
\begin{eqnarray} 
\int_{\R^d} |T^{\mu}_t g| d\widetilde{\mu} &=& \int_{\R^d} |T^{\mu}_t g^+ - T^{\mu}_t g^-| d\widetilde{\mu} \leq \int_{\R^d} |T^{\mu}_t g^+|  + |T^{\mu}_t g^-|  d\widetilde{\mu} \nonumber \\
&=& \int_{\R^d} T^{\mu}_t g^+  + T^{\mu}_t g^-  d\widetilde{\mu}  = \int_{\R^d} T^{\mu}_t |g| d\widetilde{\mu} = \int_{\R^d} |g| d\widetilde{\mu}. \label{contprop}
\end{eqnarray}
Let $f \in L^1(\R^d, \widetilde{\mu})$. Then there exists a sequence of functions $(f_n)_{n \geq 1} \subset  C_0(\R^d)$  such that $\lim_{n \rightarrow \infty }f_n = f$ in $L^1(\R^d, \widetilde{\mu})$. By \eqref{contprop}, $\widetilde{T}_t f := \lim_{n \rightarrow \infty} T^{\mu}_t f_n$ in $L^1(\R^d, \widetilde{\mu})$, $t>0$, is well-defined, independently of choice of $(f_n)_{n \geq 1}$. Moreover, by \eqref{invaproperty}, it holds
\begin{equation} \label{invnewsem}
\int_{\R^d} \widetilde{T}_t f d\widetilde{\mu} =\int_{\R^d}f d\widetilde{\mu},
\end{equation}
and by \eqref{contprop}, it follows
\begin{equation} \label{contextf}
 \int_{\R^d} |\widetilde{T}_t f| d\widetilde{\mu} \leq \int_{\R^d}|f| d\widetilde{\mu}.
\end{equation}
Assume $h \in L^1(\R^d, \mu) \cap L^1(\R^d, \widetilde{\mu}) \cap \mathcal{B}_b(\R^d)_0$ and let $(\eta_n)_{n \geq 1}$ be a standard mollifier. Then $h * \eta_n \in C_0^{\infty}(\R^d)$ for any $n \in \N$, and it follows from Lebesgue's Theorem that $\lim_{n \rightarrow \infty} h * \eta_n = h$ in $L^1(\R^d, \mu)$ and in $L^1(\R^d, \widetilde{\mu})$. Thus, there exists a subsequence $(n_k)_{k \geq 1}$ such that
$$
\lim_{k \rightarrow \infty} T^{\mu}_{t} (h * \eta_{n_k} )  = T^{\mu}_t h \;\;\;  \text{$\mu$-a.e.} \quad  \lim_{k \rightarrow \infty} \widetilde{T}_{t} (h * \eta_{n_k} )  = \widetilde{T}_t h \;\;  \text{$\widetilde{\mu}$-a.e.} 
$$
Since $\widetilde{\mu} \ll \mu$ and $T^{\mu}_{t} (h * \eta_{n_k} )=\widetilde{T}_{t} (h * \eta_{n_k} )$ \,$\widetilde{\mu}$-a.e., we obtain $T^{\mu}_t h = \widetilde{T}_t h$ \,$\widetilde{\mu}$-a.e. If $f \in L^1(\R^d, \mu) \cap L^1(\R^d, \widetilde{\mu})$, we get by Lebesgue's Theorem, 
$$
\lim_{n \rightarrow \infty} (f\cdot 1_{\overline{B}_n} \wedge n \vee -n) = f \; \text{ in $L^1(\R^d, \mu)$ and $L^1(\R^d, \widetilde{\mu})$}.
$$
Exactly as above, it then follows
\begin{equation} \label{consistency1}
T^{\mu}_t f = \widetilde{T}_t f \,\; \text{ $\widetilde{\mu}$-a.e.}
\end{equation}
Let $g \in C_0(\R^d)$. Then $T^{\mu}_s g, g \in L^1(\R^d, \mu) \cap L^1(\R^d, \widetilde{\mu})$ and so by \eqref{consistency1} for $t,s>0$
$$
\widetilde{T}_{t+s} g = T^{\mu}_{t+s} g = T^{\mu}_t T^{\mu}_s g = \widetilde{T}_t T^{\mu}_s g  = \widetilde{T}_t \widetilde{T}_s g \; \text{ $\widetilde{\mu}$-a.e.  \; hence in $L^1(\R^d, \widetilde{\mu})$.} 
$$
Using \eqref{contextf}, the latter extends to $g \in L^1(\R^d, \widetilde{\mu})$. Now let $g \in C^{\infty}_0(\R^d)$. Then $\widetilde{\mu} \ll \mu$  and \eqref{consistency1} imply
$$
|\widetilde{T}_t g-g|= |T^{\mu}_t g - g | = \left|\int_0^t T^{\mu}_s L g ds \right| \leq \int_0^{t}  T^{\mu}_s |L g|\,ds \; \text{ $\widetilde{\mu}$-a.e.}
$$
and by Fubini's Theorem, \eqref{contextf} and \eqref{consistency1}, it follows
$$
\int_{\R^d} \int_0^{t} T^{\mu}_s |L g|\,ds d\widetilde{\mu} \leq \int_0^{t} \int_{\R^d} \widetilde{T}_s |Lg| d\widetilde{\mu} ds
\leq t \int_{\R^d }|Lg| d\widetilde{\mu}.
$$
Thus
\begin{equation*} 
\| \widetilde{T}_t g -g  \|_{L^1(\R^d, \widetilde{\mu})} \leq t \| L g \|_{L^1(\R^d, \widetilde{\mu})}
\end{equation*}
and so
\begin{equation} \label{stcontin}
 \lim_{t \rightarrow 0+} \widetilde{T}_t g  = g, \; \text{ in } L^1(\R^d, \widetilde{\mu}).
\end{equation}
Using \eqref{contextf} and the denseness of $C_0^{\infty}(\R^d)$ in $L^1(\R^d, \widetilde{\mu})$, \eqref{stcontin} extends to all $g \in L^1(\R^d, \widetilde{\mu})$. Therefore, we obtain that $(\widetilde{T}_t)_{t>0}$ is a $C_0$-semigroup of contractions on $L^1(\R^d, \widetilde{\mu})$ satisfying \eqref{consistency1}. The sub-Markovian property follows from \eqref{consistency1}, the sub-Markovian property of $(T^{\mu}_t)_{t>0}$, $\widetilde{\mu} \ll \mu$ and approximation. Since $Lg\in L^1(\R^d, \mu) \cap L^1(\R^d, \widetilde{\mu})$, it follows from \eqref{consistency1} that
\begin{equation*}
\frac{\widetilde{T}_t g-g}{t}  =\frac{T^{\mu}_t g-g}{t}  =\frac{1}{t}\int_0^t T^{\mu}_s L g ds= \frac{1}{t}\int_0^t \widetilde{T}_s L g ds, \quad \text{$\widetilde{\mu}$-a.e.}
\end{equation*}
Since the above right hand side converges to  $Lg$ in $L^1(\R^d, \widetilde{\mu}) $ as  $t \rightarrow 0+$, so does the left hand side. Hence $C_0^{\infty}(\R^d)\subset D(\widetilde{L})$, $\widetilde{L}g = Lg$ \,$\widetilde{\mu}$-a.e. and by \eqref{invnewsem}
$$
\int_{\R^d} Lg\, d\widetilde{\mu} = \lim_{t\to 0+}\int_{\R^d} \frac{\widetilde{T}_t g-g}{t}\, d\widetilde{\mu} =0.
$$
Therefore, $\widetilde{\mu}$ is an infinitesimally invariant measure for $(L, C_0^{\infty}(\R^d))$. Finally, since $T^{\mu}_t f  = T^{\widetilde{\mu}}_t f$ $\widetilde{\mu}$-a.e. for any $f \in L^1(\R^d, \mu) \cap L^1(\R^d, \widetilde{\mu})$ by Theorem \ref{pathuniquenesshunt}, it follows from \eqref{consistency1} that $(T^{\widetilde{\mu}}_t)_{t>0} = (\widetilde{T}_t)_{t>0}$.
\end{proof}

\begin{theo} \label{existuniquenessinv}
Assume that $(T^{\mu}_t)_{t>0}$ is recurrent. Then $\mu$ is the unique invariant measure for $(T^{\mu}_t)_{t>0}$ in the following sense. If $\widetilde{\mu}=\widetilde{\rho} dx$ with $\widetilde{\rho} \in L^1_{loc}(\R^d)$ is an invariant measure for $(T^{\mu}_t)_{t>0}$, then there exists a constant $c>0$ such that $\widetilde{\mu}= c \mu$.
\end{theo}
\begin{proof}
By Corollary \ref{existinvariancemea}, $\mu$ is an invariant measure for $(T^{\mu}_t)_{t>0}$. By Lemma  \ref{extc0semi}, $\widetilde{\mu}$ is an infinitesimally invariant measure for $(L, C_0^{\infty}(\R^d))$. The assertion hence  follows from Theorem \ref{preinvlemma}.
\end{proof}
\text{}\\
{\bf
From now on, the uniqueness of infinitesimally invariant measures and invariant measures are in the sense of Theorem \ref{preinvlemma} and Theorem \ref{existuniquenessinv}, respectively.
}

\begin{lem} \label{fininvlem} 
Assume that either (i) or (ii)  below holds.
\begin{itemize}
\item[(i)]
There exists a finite invariant measure $\widetilde{\mu}$ for $(T^{\mu}_t)_{t>0}$. 
\item[(ii)]
$(T^{\mu}_t)_{t>0}$ is conservative and there exist a finite infinitesimally invariant measure $\widetilde{\mu}$ for $(L,C_0^{\infty}(\R^d))$.
\end{itemize}
Then $(T^{\mu}_t)_{t>0}$ is recurrent. In particular, $\mu$ is finite and is the unique infinitesimally invariant measure for $(L, C_0^{\infty}(\R^d))$ and the unique invariant measure for $(T^{\mu}_t)_{t>0}$.
\end{lem}
\begin{proof}
Assume (i). By Lemma \ref{extc0semi}, $\widetilde{\mu}$ is an infinitesimally invariant measure for $(L, C_0^{\infty}(\R^d))$. By Theorem \ref{pathuniquenesshunt}, $T^{\mu}_t f = T^{\widetilde{\mu}}_t f$ for all $f \in L^1(\R^d, \mu) \cap L^1(\R^d, \widetilde{\mu})$ and $\widetilde{\mu}$ is a finite invariant measure for $(T^{\widetilde{\mu}}_t)_{t>0}$. By Proposition \ref{fininvariant}, $(T^{\widetilde{\mu}}_t)_{t>0}$ is recurrent, hence $(T^{\mu}_t)_{t>0}$ is recurrent by Theorem \ref{pathuniquenesshunt}. The rest of the assertion follows from Theorem \ref{preinvlemma}, Corollary \ref{existinvariancemea} and Theorem \ref{existuniquenessinv}.\\
Assume (ii). By Theorem \ref{pathuniquenesshunt}, $(T^{\widetilde{\mu}}_t)_{t>0}$ is conservative. Thus, $(T^{\widetilde{\mu}}_t)_{t>0}$ is recurrent by Proposition \ref{fininvariant}, hence $(T^{\mu}_t)_{t>0}$ is recurrent by Theorem \ref{pathuniquenesshunt}. The rest of the assertion follows from Theorem \ref{preinvlemma}, Corollary \ref{existinvariancemea} and Theorem \ref{existuniquenessinv}.
\end{proof}
\text{}\\
The following results deliver sufficient conditions for non-existence of finite invariant measures and finite infinitesimally invariant measures. 
\begin{prop}\label{nonexinv}
Suppose that at least one of the semigroups $(T^{\mu}_t)_{t>0}$ or $(T'^{,\mu}_t)_{t>0}$ is non-conservative. Then, there does not exist 
a finite invariant measure for $(T^{\mu}_t)_{t>0}$.
\end{prop}
\begin{proof}
Suppose to the contrary that there exists a finite invariant measure $\widetilde{\mu}$ for $(T^{\mu}_t)_{t>0}$. Then by Lemma \ref{fininvlem}, $(T^{\mu}_t)_{t>0}$ is recurrent, which contradicts the above assumption by Corollary \ref{existinvariancemea}.
\end{proof}

\begin{prop} \label{prop1.22}
Assume that there exists an infinite infinitesimally invariant measure $\widehat{\mu}$ for $(L, C_0^{\infty}(\R^d))$. Then the following properties are satisfied.
\begin{itemize}
\item[(i)]
There does not exist a finite invariant measure for $(T^{\mu}_t)_{t>0}$.
\item[(ii)]
If $(T^{\mu}_t)_{t>0}$ is conservative, then there is no finite infinitesimally invariant measure for $(L, C_0^{\infty}(\R^d))$. 
\end{itemize}
\end{prop}
\begin{proof}
(i) Suppose there exists a finite invariant measure $\widetilde{\mu}$ for $(T^{\mu}_t)_{t>0}$. 
Then by Lemma \ref{extc0semi}, $\widetilde{\mu}$ is an infinitesimally invariant measure for $(L, C_0^{\infty}(\R^d))$. 
Moreover, $(T^{\mu}_t)_{t>0}$ is recurrent  by Lemma \ref{fininvlem}. Therefore, by Theorem \ref{preinvlemma}, $\widehat{\mu}$ is finite, which is contradiction.\\
(ii) If there exists a finite infinitesimally invariant measure $\widetilde{\mu}$ for $(L, C_0^{\infty}(\R^d))$, then
by Lemma \ref{fininvlem}, $(T^{\mu}_t)_{t>0}$ is recurrent. By Theorem  \ref{preinvlemma}, $\widehat{\mu}$ equals $\mu$ up to a multiplicative constant, which is contradiction.
\end{proof}
\text{}\\
As a direct consequence of Proposition \ref{nonexinv}, Proposition \ref{prop1.22} and Lemma \ref{fininvlem}, we obtain the following result which characterizes the non-existence of finite invariant measures for $(T^{\mu}_t)_{t>0}$.
\begin{cor} \label{carinva}
$(T^{\mu}_t)_{t>0}$ has no finite invariant measure, if and only if there exists an infinite infinitesimally invariant measure for $(L, C_0^{\infty}(\R^d))$ or $(T^{\mu}_t)_{t>0}$ is non-conservative or $(T'^{, \mu}_t)_{t>0}$ is non-conservative.
\end{cor}
\begin{rem}\label{explicitcritinv}
\rm Combining Theorems \ref{preinvlemma} and \ref{existuniquenessinv} with explicit recurrence results that hold under condition {\bf (H)} leads to explicit criteria for existence and uniqueness of (infinitesimally) invariant measures. In particular, in the case where $\mu$ is finite,  
(explicit) recurrence results can be derived by (explicit) conservativeness results for $(T^{\mu}_t)_{t>0}$ or $(T'^{, \mu}_t)_{t>0}$
according to Proposition \ref{fininvariant} and Lemma \ref{equicondi32} under the mere assumption {\bf (a)} of \cite{LST20} (cf. paragraph right before Proposition \ref{fininvariant}). Explicit conservativeness or recurrence criteria in the here considered framework can be found for instance in \cite{GT2}, \cite{LST20} and \cite[Corollary 15]{GT17}, which is the basis for (iii)(b) below. We mention here several different kinds of examples, which are all concluded with the additional application of Theorems \ref{preinvlemma} and \ref{existuniquenessinv} (and Theorem \ref{cosemirec} for (i) and (ii) below). In all these examples, according to \cite[Corollary 2.2]{St99}, additionally $L^1$-uniqueness holds:\\
(i) As a consequence of \cite[Proposition 3.40, Theorem 3.38(iii), Corollary 3.41]{LST20} (which imply recurrence of $(T^{\mu}_t)_{t>0}$), we obtain the following:\\[3pt]
{\it Suppose that for some $N_0 \in \N$ and a constant $c \geq 0$, there exists $V \in C(\R^d) \cap C^2(\R^d \setminus \overline{B}_{N_0})$ with $\inf_{\partial B_r} V \rightarrow \infty$ as $r \rightarrow \infty$, such that
\begin{equation} \label{lapynovprop}
LV \leq -c \quad \text{ a.e. on } \; \R^d \setminus \overline{B}_{N_0}.
\end{equation}
Then $(T^{\mu}_t)_{t>0}$ and $(T'^{, \mu}_t)_{t>0}$ are recurrent, and $\mu$ is the unique infinitesimally invariant measure for $(L, C_0^{\infty}(\R^d))$ and $(L'^{, \mu}, C_0^{\infty}(\R^d))$ and the unique invariant measure for $(T^{\mu}_t)_{t>0}$ and $(T'^{, \mu}_t)_{t>0}$.  In particular, if $c>0$, then $\mu$ is finite by \cite[Proposition 3.47]{LST20} (see also \cite[Theorem 2]{BRS} for the original result) and from \cite[Corollary 3.48]{LST20}.
An example of such a $V$ is given by $V(x):= \frac12 \ln(\|x\|^2 \vee N_0^2)+2$, $x \in \R^d$  for some $N_0 \in \N$. Then,
\eqref{lapynovprop} is fulfilled for the given $V$, if}
\begin{eqnarray*} 
\label{layporec}
-\frac{\langle A(x)x, x \rangle}{ \left \| x \right \|^2 }+ \frac12\mathrm{trace}A(x)+ \big \langle \mathbf{G}(x), x \big \rangle \leq -c\|x\|^2, \quad \text{ for a.e. } x \in \R^d \setminus \overline{B}_{N_0}.
\end{eqnarray*}
Note: \cite[Theorem 4.4.1]{BKRS} considers a criterion similar to \eqref{lapynovprop} with $c=0$ to obtain uniqueness of infinitesimally invariant measures for $(L, C_0^{\infty}(\R^d))$. Indeed, \cite[Theorem 4.4.1]{BKRS} is proven explicitly through a direct calculation using auxiliary functions and the strong maximum principle, whereas we use the general concept Theorem \ref{preinvlemma} and a recurrence criterion presented in \cite{LST20}. In particular, we obtain not only that $\mu$ is the unique infinitesimally invariant measure for $(L, C_0^{\infty}(\R^d))$, but also for $(L'^{, \mu}, C_0^{\infty}(\R^d))$, since $(T'^{, \mu}_t)_{t>0}$ is also recurrent by Theorem \ref{cosemirec}. Moreover, we also obtain that $\mu$ is the unique invariant measure for $(T^{\mu}_t)_{t>0}$ and $(T'^{, \mu}_t)_{t>0}$ by Theorem \ref{existuniquenessinv}.\\
(ii) Similarly to (i) the following follows from \cite[Corollaries 3.41, 3.48, Proposition 3.47, and Theorem 3.38(iii)]{LST20}:\\
{\it Assume $d=2$. Let $\Psi_1, \Psi_2 \in C(\R^2)$ with $\Psi_1(x), \Psi_2(x)>0$ 
for all $x \in \R^2$, $N_0 \in \N$ and $Q=(q_{ij})_{1 \leq i,j \leq 2}$ be a matrix of measurable functions such that  
$Q^T(x)Q(x)=id$ for all $x\in \R^2\setminus \overline{B}_{N_0}$. 
Suppose that the diffusion matrix $A$ (satisfying {\bf (H)}) has the following form
$$
A(x) = Q^T(x) \begin{pmatrix}
\Psi_1(x)  & 0 \\ 
0 & \Psi_2(x)
\end{pmatrix}Q(x) \; \;\text{ for all $x\in \R^2\setminus \overline{B}_{N_0}$,}
$$
and that there exists a constant $c \geq 0$, such that
\begin{equation} 
\label{eq:3.2.1.20}
\frac{|\Psi_1(x)-\Psi_2(x)|}{2}  + \langle \mathbf{G}(x),x  \rangle \leq -c\|x\|^2
\end{equation}
for a.e. $x\in \R^2\setminus \overline{B}_{N_0}$. Then $(T^{\mu}_t)_{t>0}$ and $(T'^{, \mu}_t)_{t>0}$ are recurrent, and $\mu$ is the unique infinitesimally invariant measure for $(L, C_0^{\infty}(\R^d))$ and $(L'^{, \mu}, C_0^{\infty}(\R^d))$ and the unique invariant measure for $(T^{\mu}_t)_{t>0}$ and $(T'^{, \mu}_t)_{t>0}$.  In particular, if $c>0$, then $\mu$ is finite.}\\
(iii) Consider assumption {\bf (H$^{\prime}$)}. Then $(T^{\mu}_t)_{t>0}$ and $(T'^{, \mu}_t)_{t>0}$ are recurrent and $\mu$ is the
unique infinitesimally invariant measure for $(L=\frac12\text{trace}(A \nabla^2) +\langle \beta^{\rho, A}+\mathbf{B}, \nabla \rangle=L^{0,\mu}+\langle \mathbf{B}, \nabla\rangle, C_0^{\infty}(\R^d))$ and $(L'^{, \mu}=\frac12\text{trace}(A \nabla^2) +\langle \beta^{\rho, A}-\mathbf{B}, \nabla \rangle=L^{0,\mu}-\langle \mathbf{B}, \nabla\rangle, C_0^{\infty}(\R^d))$ and the unique invariant measure for $(T^{\mu}_t)_{t>0}$ and $(T'^{, \mu}_t)_{t>0}$, if one of the following conditions (a)-(c) is satisfied:
\begin{itemize}
\item[(a)] {\it (Cf. \cite[Proposition 3.41]{LST20})
$$
a_n:=\int_1^n r\Big (\int_{B_r} \Big (\frac{\langle A(x)x, x \rangle}{\|x\|^2} + | \langle \mathbf{B}(x),x  \rangle  | \Big )\mu(dx)\Big )^{-1}dr\longrightarrow \infty \ \text{ as } \ n\to \infty
$$
and }
$$
\lim _{n\to \infty}\frac{\ln\big (\int_{B_n}  |\langle \mathbf{B}(x),x  \rangle  | \mu(dx)\vee 1\big )}{a_n}=0.
$$
\item[(b)] {\it (Cf. \cite[Proposition 3.31(i)]{LST20}) $\mu$ is finite and there exist constants $M>0$ and $N_0\in \N$, such that 
$$
\frac{\langle A(x)x,x\rangle}{\|x\|^2}+|\langle\mathbf{B}(x),x\rangle| \leq M\|x\|^{2} \ln(\|x\|+1),
$$
for a.e. $x\in \R^d\setminus \overline{B}_{N_0}$.}
\item[(c)] {\it (Cf. \cite[Proposition 2.15(i), Corollary 2.16(i)]{LST20}) $\mu$ is finite and }
$$
a_{ij}\in L^1 (\R^d ,\mu), \ 1\le i,j\le d, \ \mathbf{B}\in L^1 (\R^d, \R^d, \mu).
$$
Note: Condition (c) is used in \cite[Theorem 4.1.6(ii)]{BKRS} to prove uniqueness of infinitesimally invariant probability measures among all infinitesimally invariant measures (cf. \cite[Remark 4.17]{BKRS}) in the special case where $\mathbf{B}=\mathbf{G}-\beta^{\rho, A}$. Here we consider {\bf (H$^{\prime}$)} and $\mu$ finite, which can of course be replaced by the abstract assumption {\bf (H)} and $\mu$ finite as in \cite[Theorem 4.1.6]{BKRS}. We call the latter assumption abstract as then $\rho$, except its local regularity properties is unknown, so that the condition $\mathbf{G}-\beta^{\rho, A}\in L^1 (\R^d, \R^d, \mu)$ is in general not checkable. In contrast to \cite{BKRS}, we do not need an individual proof for (c), since it follows from the general concept of recurrence, which seems not to be disclosed in \cite{BKRS}. Moreover also \cite[Theorem 4.1.6(iii)]{BKRS} follows immediately from our results since the  classical Lyapunov condition there is just a sufficient condition for conservativeness of $(T^{\mu}_t)_{t>0}$ (\cite[Corollary 2.16(ii)]{LST20}), which then by the finiteness of $\mu$ implies recurrence (Proposition \ref{fininvariant}).
\end{itemize}
\end{rem}
\text{}\\
In the following example, we discuss different aspects of Remark  \ref{explicitcritinv}(i) and in particular the sharpness of condition \eqref{lapynovprop} in relation to the values of $c\geq0$.

\begin{exam} \label{exampleinva} 
\rm
(i) Let $m \in \R$ and consider 
$$
\rho(x):=\left\{\begin{matrix}
&1&  \text{on }\, \overline{B}_1, \\ 
&\|x\|^{m}&  \text { on }\, \R^d \setminus \overline{B}_1.
\end{matrix}\right.
$$ 
Then $\rho \in  H_{loc}^{1,\infty}(\R^d)\cap C(\R^d) \subset H^{1,p}_{loc}(\R^d) \cap C(\R^d)$. Let $A=id$ and $\mathbf{G}=\frac{\nabla \rho}{2 \rho}$.  Then for any $f \in C_0^{\infty}(\R^d)$,
$Lf = \frac12\Delta f$ on $B_1$,
$$
Lf(x)  =\frac12 \Delta f(x) + \big \langle \frac{mx}{2 \|x\|^2}, \nabla f(x) \big \rangle \quad \text{ for all $x \in \R^d \setminus \overline{B}_1$.}
$$
and $\mu=\rho dx$ is an infinitesimally invariant measure for $(L, C_0^{\infty}(\R^d))$. By \cite[Proposition 1.10(c)]{St99}, $(T^{\mu}_t)_{t>0}$ is conservative and $\mu$ is an invariant measure for $(T^{\mu}_t)_{t>0}$. Let $V(x)= 1+\|x\|^2$. Then
\[
LV(x)=\text{trace}(A(x))+2 \langle \mathbf{G}(x), x \rangle =d+m \quad \text{  for all $x \in \R^d \setminus \overline{B}_1$. }
\]
Here $(\mathcal{E}^{0, \mu}, D(\mathcal{E}^{0, \mu}))$ is the symmetric Dirichlet form defined as the closure of
$$
\mathcal{E}^{0,\mu}(f,g)=\frac12\int_{\R^d} \langle \nabla f, \nabla g \rangle d\mu, \quad f,g \in C_0^{\infty}(\R^d), 
$$
on $L^2(\R^d, \mu)$ and by Proposition \ref{prop1:9}, $(T^{\mu}_t)_{t>0}$ considered as semigroup on $L^2(\R^d, \mu)$ is associated with $(\mathcal{E}^{0, \mu}, D(\mathcal{E}^{0, \mu}))$. If $m+d \neq 0$, then for any $r>1$,
$$
\mu(B_r) = \int_{B_r} d\mu  = d\cdot dx(B_1)\int_1^r |s|^{m+d-1} ds +\mu(B_1) = d \cdot dx(B_1)\frac{1}{m+d}(r^{m+d}-1) + dx(B_1).
$$
If $m+d=0$, then $\mu(B_r)=d \cdot dx(B_1)\ln r + dx(B_1)$. Thus, $\mu$ is finite, if and only if  $m<-d$. If $m=-d$, then $LV(x) =0$ for all $x \in \R^d \setminus \overline{B}_1$, but $\mu$ is not finite.
Moreover, one can calculate
\begin{equation} \label{recintcon}
\int_{1}^{\infty} \frac{r}{\mu(B_r)} dr =\infty
\end{equation}
if and only if $m+d \leq 2$. By \cite[Theorem 3]{Stu1} and \cite[Theorem 4.1]{Sturm98} (or alternatively by \cite[Theorem 21]{GT2} and Proposition \ref{strictirreducibilityprop}), $(T^{\mu}_t)_{t>0}$ is recurrent if $m+d \leq 2$. If we choose $m:=2-d$, then $\mu$ is infinite and $LV(x) =2>0$ for all $x \in \R^d \setminus B_1$, but nonetheless $(T^{\mu}_t)_{t>0}$ is recurrent. 
Assume $m=0$. Then it is well-known that $(T^{\mu}_t)_{t>0}$ is recurrent if $d=2$ and transient if $d \geq 3$ (which can be reconfirmed using \cite[Corollary 2.9]{Stu2}, \cite[Theorem 4.1]{Sturm98} and \eqref{recintcon}).\\
On the other hand, if we choose  for some $N_0 \in \N$, $V \in C(\R^d) \cap C^2(\R^d \setminus \overline{B}_{N_0})$ defined as in Remark \ref{explicitcritinv}(i), then
$$
LV(x)=\frac{-2+d +m}{2\|x\|^2}\;\quad \text{ for all  } x \in \R^d \setminus \overline{B}_{N_0},
$$
hence $LV \leq 0$ on  $\R^d \setminus \overline{B}_{N_0}$, if and only if $m+d \leq 2$.
\\
\centerline{}
(ii) In the situation of (i), consider again $V(x):=1+\|x\|^2$, $x \in \R^d$ and $m+d =0$.  Then $LV  = 0$ on $\R^d \setminus \overline{B}_{1}$ and by Remark \ref{explicitcritinv}(i), $\mu$ is
the unique infinitesimally invariant measure for $(L, C_0^{\infty}(\R^d))$ and the unique  invariant measure for $(T^{\mu}_t)_{t>0}$.  Let $\psi \in H^{1,p}_{loc}(\R^d) \cap C(\R^d)$ be arbitrarily given, such that $\psi(x)>0$ for all $x \in \R^d$. Then $\frac{\rho}{\psi}\in H^{1,p}_{loc}(\R^d)\cap C(\R^d)$,
\begin{eqnarray}\label{preinvltilde}
\int_{\R^d} \Big(\frac{\rho}{\psi} \Big ) Lf \cdot \psi dx = 0, \qquad \forall f \in C_0^{\infty}(\R^d),
\end{eqnarray}
and
\begin{eqnarray}\label{preinvltilde0rec}
\frac{\rho}{\psi} L V = 0,  \;\; \text{ a.e. on $\R^d\setminus \overline{B}_{1}$.}
\end{eqnarray}
Let $\mathcal{L}:= \frac{\rho}{\psi} L$ on $C_0^{\infty}(\R^d)$, i.e.
$$
\mathcal{L}f:= \frac12 \text{trace}(\frac{\rho}{\psi}A\nabla^2 f) + \langle \frac{\rho}{\psi} \mathbf{G}, \nabla f  \rangle, \qquad f \in C_0^{\infty}(\R^d).
$$
 Then the coefficients of $\mathcal{L}$ satisfy the assumption {\bf (H)} and $\psi dx$ is an infinitesimally invariant measure for $(\mathcal{L}, C_0^{\infty}(\R^d))$. Thus using \eqref{preinvltilde} and \cite[Theorem 1.5]{St99}, there exists a closed extension $(\overline{\mathcal{L}}, D(\overline{\mathcal{L}}))$ of $(\mathcal{L}, C_0^{\infty}(\R^d))$ on $L^1(\R^d, \psi dx)$ which generates a sub-Markovian $C_0$-semigroup of contractions $(\mathcal{T}_t)_{t>0}$, $r \in [1, \infty)$ (see Section \ref{section1}). Using \eqref{preinvltilde0rec} and Remark \ref{explicitcritinv}(i), $(\mathcal{T}_t)_{t>0}$ is recurrent and $\psi dx$ is possibly finite and the unique infinitesimally invariant measure for $(\mathcal{L}, C_0^{\infty}(\R^d))$  and the unique  invariant measure for $(\mathcal{T}_t)_{t>0}$.
\end{exam}
In the following example we further study the \lq\lq  sharpness\rq\rq\ of  the Lyapunov conditon. This example could also be considered with an additional non-sectorial weakly $\mu$-divergence free perturbation. However, to keep it simple, we just consider the symmetric case of distorted Brownian motion, which is for itself meaningful.
\begin{exam} \label{exampleinva2} \rm
(i) Let $p>d$ and $\gamma \in (0,1)$ be such that $p(1-\gamma) <d$. Define $\psi_0(x):=\|x\|^{\gamma}$, $x \in B_{1/4}$. Then $\psi_0 \in H^{1,p}(B_{1/4}) \cap C(\overline{B}_{1/4})$. Extend $\psi_0 \in H^{1,p}(B_{1/4})$ to a positive function $\psi \in H^{1,p}(B_{1/2}) \cap C_0(B_{1/2})$ and define (cf. \cite[Example 3.30(ii)]{LST20})
$$
\phi(x) :=1+\sum_{k=1}^{\infty} \psi(x-k\mathbf{e}_1), \quad \text{ for }x \in \R^d, \quad \mathbf{e}_1:=(1,0,\ldots,0).
$$ 
Then $\phi \in H^{1,p}_{loc}(\R^d) \cap C(\R^d)$ is bounded below and above by strictly positive constants and $\nabla \phi$ is unbounded in a neighborhood of infinitely many isolated points. More precisely, it holds for each $k \in \N$
$$
\phi(x)=1+\|x-k\mathbf{e}_1\|^{\gamma}, \text{ for all $x \in B_{1/4}(k\mathbf{e}_1)$,} 
$$
and
$$\nabla \phi (x)= \frac{\gamma}{\|x-k\mathbf{e}_1\|^{1-\gamma}} \frac{x-k\mathbf{e}_1}{\|x-k\mathbf{e}_1\|} \;\; \text{ a.e. on } B_{1/4}(k\mathbf{e}_1).
$$
Now let 
$$
\rho(x):=\exp(-\|x\|^2 \phi(x)), \quad \text{ for }x \in \R^d.
$$
Then $\mu=\rho dx$ is finite and $\rho \in H^{1,p}_{loc}(\R^d) \cap C(\R^d)$ with $\rho(x)>0$ for all $x\in \R^d$. Consider $(L, C_0^{\infty}(\R^d))$ given by
$$
Lf = \frac12 \Delta f +\langle \mathbf{G}, \nabla f \rangle, \quad f \in C_0^{\infty}(\R^d),\quad \text{with }\ \mathbf{G}=\frac{\nabla \rho}{2\rho}.
$$
Then $\mu$ is an infinitesimally invariant measure for $(L, C_0^{\infty}(\R^d))$.  By \cite[Proposition 1.10(a)]{St99} or \cite[Proposition 2.15(i)]{LST20}, $\mu$ is $(T^{\mu}_t)_{t>0}$-invariant (see Remark \ref{equivconinv}(ii)),  hence by Proposition \ref{fininvariant} $(T^{\mu}_t)_{t>0}$ is recurrent. 
Thus, by Theorems \ref{preinvlemma} and \ref{existuniquenessinv}, $\mu$ is the unique infinitesimally invariant measure for $(L, C_0^{\infty}(\R^d))$ and the unique invariant measure for $(T^{\mu}_t)_{t>0}$, respectively.
Let $k \in \N$ be fixed and define
$$
\mathbf{F}^k(x):=\frac{\gamma}{\|x-k\mathbf{e}_1\|^{1-\gamma}} \frac{x-k\mathbf{e}_1}{\|x-k\mathbf{e}_1\|}, \quad \text{ for  } x \in B_{1/4}(k\mathbf{e}_1) \setminus \{k \mathbf{e}_1 \}.
$$
Then $\mathbf{F}^k(x)=\nabla \phi(x)$ a.e. on $B_{1/4}(k\mathbf{e}_1)$. For each $t \in (-1/4,1/4) \setminus \{0\}$
\begin{eqnarray*}
\left \langle \mathbf{F}^k((t+k)\mathbf{e}_1), (t+k)\mathbf{e}_1 \right \rangle &=& \left \langle\frac{\gamma}{|t|^{1-\gamma}} \frac{t\mathbf{e}_1}{|t|}, (t+k)\mathbf{e}_1 \right \rangle
=\frac{\gamma}{|t|^{1-\gamma}} (|t| +k \frac{t}{|t|}) \\
&=& \gamma |t|^{\gamma} +\frac{k\gamma}{|t|^{1-\gamma}}\frac{t}{|t|} \longrightarrow -\infty \quad \text{ as $t \rightarrow 0-$}.
\end{eqnarray*}
Therefore, there exists a sequence $(x_n^k)_{n \geq 1}$ in $B_{1/4}(k\mathbf{e}_1) \setminus \{k \mathbf{e}_1 \}$ such that $\lim_{n \rightarrow \infty}x_n^k = k\mathbf{e}_1$ and
\begin{equation}\label{26}
\langle \mathbf{F}^k(x_n^k), x_n^k \rangle \longrightarrow -\infty \quad \text{ as $n \rightarrow \infty$.}
\end{equation}
Since
$$
\frac{\nabla \rho}{2\rho}(x) = \frac12 \nabla (\ln \rho)(x) =\frac12 (-2 \phi(x) x -\|x\|^2 \nabla \phi(x)) =-\phi(x) x-\frac{\|x\|^2}{2} \mathbf{F}^k(x), \quad \text{a.e. on $B_{1/4}(k\mathbf{e}_1)$},
$$ 
$$
\langle \mathbf{G}(x), x \rangle = -\phi(x)\| x\|^2-\frac{\|x\|^2}{2} \langle \mathbf{F}^k(x), x\rangle \quad \text{a.e. on $B_{1/4}(k\mathbf{e}_1)$}.
$$
Note that the function $x \mapsto -\phi(x)\| x\|^2-\frac{\|x\|^2}{2} \langle \mathbf{F}^k(x), x\rangle$ is continuous on $B_{1/4}(k\mathbf{e}_1) \setminus \{k \mathbf{e}_1 \}$ and that by \eqref{26}
$$
-\phi(x_n^k)\| x_n^k\|^2-\frac{\|x_n^k\|^2}{2} \langle \mathbf{F}^k(x_n^k), x_n^k\rangle \longrightarrow \infty \quad \text{ as } n \rightarrow \infty.
$$
Thus for arbitrarily given $M>0$, there exists an open ball $V_M^k \subset B_{1/4}(k\mathbf{e}_1) \setminus \{ k\mathbf{e}_1 \}$ such that
$$
\langle \mathbf{G}(x), x \rangle \geq M, \quad \text{ a.e. on } V_M^k,
$$
which means that there is no essential upper bound for $\langle \mathbf{G}(x),x \rangle$ on $B_{1/4}(k\mathbf{e}_1)$.
If we choose $V \in C^2(\R^d)$ such that
$$
V(x)=1+\|x\|^2, \quad \text{ or } \; V(x)=\ln(1+\|x\|^2), 
$$
then 
$$
LV(x) =d+2\langle \mathbf{G}(x), x \rangle
$$
or
$$
LV(x) = \frac{(d-2)\|x\|^2 +d}{(\|x\|^2+1)^2}+\frac{2\langle \mathbf{G}(x),x \rangle}{\|x\|^2+1},
$$
respectively. In both cases, it does not hold that for some constants $K\ge 0$ and $N_0 \in \N$
$$
LV \leq -K, \quad \text{ a.e. on } \R^d \setminus \overline{B}_{N_0}
$$
and it might be rather difficult to find a Lyapunov function V with the latter property.\\[3pt]
(ii)\,Let $m \in \R$ and
$$
\widetilde{\rho}(x):=\left\{\begin{matrix}
&1&  \text{on }\, \overline{B}_1, \\ 
&\|x\|^{m}&  \text { on }\, \R^d \setminus \overline{B}_1,
\end{matrix}\right.
$$ 
$\rho=\phi \widetilde{\rho}$, $\mu=\rho dx$, where $\phi$ is defined as in (i). Since $\phi,\widetilde{\rho} \in H^{1, p}_{loc}(\R^d) \cap C(\R^d)$, we have $\rho \in H^{1,p}_{loc}(\R^d) \cap C(\R^d)$. Let $(L, C_0^{\infty}(\R^d))$ be given by
$$
Lf = \frac12 \Delta f +\langle \mathbf{G}, \nabla f \rangle, \quad f \in C_0^{\infty}(\R^d),\quad \text{with }\ \mathbf{G}=\frac{\nabla \rho}{2\rho}=\frac{\nabla \phi}{2\phi}+\frac{\nabla \widetilde{\rho}}{2\widetilde{\rho}}.
$$
Since $\phi$ is globally bounded above and below by strictly positive constants, according to the calculations in Example \ref{exampleinva}(i), $\mu$ is finite if and only if $m+d<0$ and there exists a constant $c>0$ independent of $r$ such that
$\mu(B_r) \leq c r^{(m+d) \vee 1}$ for any $r>1$. Thus by \cite[Proposition 3.31]{LST20}, $(T^{\mu}_t)_{t>0}$ is conservative and $\mu$ is an invariant measure for $(T^{\mu}_t)_{t>0}$. Consequently,  by Proposition \ref{prop1.22}(ii) there exists no finite infinitesimally invariant measure for $(L, C_0^{\infty}(\R^d))$  if $m+d \geq 0$.  Now assume that $0\leq m+d \leq 2$. Then $\mu$ is infinite and since \eqref{recintcon} holds, it follows from \cite[Theorem 3]{Stu1} and \cite[Theorem 4.1]{Sturm98} (or alternatively from \cite[Theorem 21]{GT2} and Proposition \ref{strictirreducibilityprop}) that $(T^{\mu}_t)_{t>0}$ is recurrent. Hence by Theorems \ref{preinvlemma} and \ref{existuniquenessinv}, $\mu$ is the unique infinitesimally invariant measure for $(L, C_0^{\infty}(\R^d))$ and the unique invariant measure for $(T^{\mu}_t)_{t>0}$, respectively. But similarly to (i), we can observe that for each $k \in \N$, $\langle \mathbf{G}(x), x \rangle$ has no essential upper bound on $B_{1/4}(k\mathbf{e}_1)$. For instance, choose $V$ explicitly as in Remark \ref{explicitcritinv}(i).
Then 
\begin{equation*}
LV(x) =\frac{d-2}{ 2\left \| x \right \|^2 }+ \frac{\big \langle \mathbf{G}(x), x \big \rangle}{\|x\|^2} \leq 0, \quad \text{ for a.e. $x \in \R^d \setminus \overline{B}_{N_0}$}
\end{equation*}
does not hold and it may be rather difficult to find a general $V \in C^2(\R^d \setminus \overline{B}_{N_0}) \cap C(\R^d)$ satisfying \eqref{lapynovprop}.
\end{exam}
\begin{rem}\label{dxunique}\rm
Let $A=id$ and $\mathbf{G}=0$. Then the Lebesgue measure $dx$ is the unique infinitesimally invariant measure for $(\frac12{\Delta}, C_0^{\infty}(\R^d))$ and the unique invariant measure for $(T^{\mu}_t)_{t>0}$. 
Indeed, if $\widetilde{\mu}$ is an infinitesimally invariant measure for $(\frac12{\Delta}, C_0^{\infty}(\R^d))$, then it is well-known (use for instance \cite[Corollaries 2.10, 2.11]{BKR2}, \cite[Theorem 3, page 334]{Ev11}, and \cite[Theorem 3.1]{A01}) that $\widetilde{\mu}= dx$ up to a multiplicative constant. Hence $\widetilde{\mu}=\mu=dx$ up to a multiplicative constant.
Then, since $(T^{\mu}_t)_{t>0}=(T'^{, \mu}_t)_{t>0}$ and $(T^{\mu}_t)_{t>0}$ is conservative, we get by Lemma \ref{equicondi32}(i) that $dx$ is an invariant measure for $(T^{\mu}_t)_{t>0}$.
Since any invariant measure for $(T^{\mu}_t)_{t>0}$ is an infinitesimally invariant measure for $(\frac12 \Delta, C_0^{\infty}(\R^d))$ by Lemma \ref{extc0semi}, uniqueness of the invariant measure $dx$ follows. As it is well-known $(T^{\mu}_t)_{t>0}$ is recurrent if $d=2$ and transient if $d \geq 3$.
Therefore, recurrence is not a necessary condition for uniqueness of infinitesimally invariant measures holding together with existence and uniqueness of invariant measures.
\end{rem}

\medskip

\section{Further examples, counterexamples and applications}\label{section5exam}
\subsection{Examples and counterexamples for uniqueness of (infinitesimally) invariant measures}
\begin{exam} \label{exam:3.1}
\rm
Let $A=id$, $\rho=1$, $\mu=dx$ and $\mathbf{G}=\mathbf{c}$, where $\mathbf{c}=(c_1, \dots, c_d) \in \R^d$ is a non-zero constant vector field on $\R^d$. Then {\bf (H)} holds and we have
$$
L f =L^{0, \mu} f + \langle \mathbf{c}, \nabla f \rangle=\frac12\Delta f+\langle \mathbf{c}, \nabla f \rangle, \quad f \in C_0^{\infty}(\R^d).
$$
$\mu=dx$ is an infinitesimally invariant measure for $(L, C_0^{\infty}(\R^d))$ since $\beta^{1, id}=0$ and
$$
\int_{\R^d} \langle \mathbf{c}, \nabla f \rangle dx = 0, \quad \forall f \in C_0^{\infty}(\R^d).
$$
$(T^{0, \mu}_t)_{t>0}$ is a version of the transition semigroup of Brownian motion and is hence recurrent if $d=2$, and transient if $d \geq 3$ (see Example \ref{exampleinva}(i)). Let $\widetilde{\rho}(x):=1+e^{2\langle \mathbf{c}, x \rangle}$, $x \in \R^d$. Then $\widetilde{\rho} \in H^{1,p}_{loc}(\R^d) \cap C(\R^d)$ with $\widetilde{\rho}(x)>0$ for all $x \in \R^d$. Set $\widetilde{\mu}:=\widetilde{\rho} dx$. Then since $\beta^{\widetilde{\rho}-1, id}$= $\mathbf{c}$, $\widetilde{\mu}$ is an infinitesimally invariant measure for $(L,C_0^{\infty}(\R^d))$. By \cite[Corollary 2.16(iii)]{LST20}, $(T^{\mu}_t)_{t>0}$ and  $(T^{\widetilde{\mu}}_t)_{t>0}$ are conservative, hence there is no finite infinitesimally invariant measure for $(L, C_0^{\infty}(\R^d))$ by Proposition \ref{prop1.22}(ii). Using \cite[Proposition 2.15(iii)]{LST20}, we can see that $\mu$ is an invariant measure for $(T^{\mu}_t)_{t>0}$, and $\widetilde{\mu}$ is an invariant measure for $(T^{\widetilde{\mu}}_t)_{t>0}$. 
Moreover, by Theorem \ref{pathuniquenesshunt}, for every $t>0$ it holds that 
\begin{equation} \label{uniqueness}
T^{\mu}_t f = T^{\widetilde{\mu}}_t f, \quad \text{$dx$-a.e.} \;\text{ for all } f \in  L^1(\R^d, dx) \cap L^1(\R^d, \widetilde{\mu}) \text{ and } t>0.
\end{equation}
By \eqref{uniqueness}, $(T^{\mu}_t)_{t>0}$ and $(T^{\widetilde{\mu}}_t)_{t>0}$ have hence two distinct invariant measures, $\mu$ and $\widetilde{\mu}$ which are not expressed by a constant multiple of each other. Using Theorem \ref{preinvlemma}, we have that $(T^{\mu}_t)_{t>0}$ and $(T^{\widetilde{\mu}}_t)_{t>0}$ cannot  be recurrent, hence they are both transient by Proposition \ref{strictirreducibilityprop}. On the other hand,
 by Theorem \ref{theo1.12}, for every $t>0$
$$
T'^{, \widetilde{\mu}}_t\big (\frac{1}{1+e^{2\langle c, x \rangle}}\big )=\frac{1}{1+e^{2\langle c, x \rangle}}
$$
and so we have found an explicit example for a bounded co-excessive function which is not a constant function. 
\end{exam}

\begin{exam} \label{twoprebothfinite}
\rm
Consider the situation of Example \ref{exam:3.1} with $d=2$. Then $\mu=dx$ is an infinitesimally invariant measure for $(L, C_0^{\infty}(\R^2))$ and it holds
$$
L f= L^{0, \mu} f +\langle \mathbf{G}-\beta^{1,id}, \nabla f \rangle= \frac12 \Delta f+ \langle \mathbf{c}, \nabla f \rangle, \quad f \in C_0^{\infty}(\R^2).
$$
By Example \ref{exampleinva}(i), $(T^{0, \mu}_t)_{t>0}$ is recurrent, but $(T^{\mu}_t)_{t>0}$ is transient by Example \ref{exam:3.1}.
Moreover, it holds
$$
 \left |\int_{\R^2} \langle \mathbf{c}, \nabla f \rangle g dx \right | \leq \|\mathbf{c}\| \|f\|_{H^{1,2}(\R^2)} \|g \|_{H^{1,2}(\R^2)}, \quad \text{ for all } f,g \in C_0^{\infty}(\R^2),
$$
hence, the generalized Dirichlet form $(\mathcal{E}, D(L^{\mu}_2))$ associated to $(T^{\mu}_t)_{t>0}$ on $L^2(\R^2, \mu)$ satisfies the {\it weak sector condition}, i.e.
$$
\sup_{u, v \in  C_0^{\infty}(\R^2) \setminus \{0 \} } \frac{|\int_{\R^2} Lu \cdot v d\mu|}{\mathcal{E}_1^{0, \mu}(u,u)^{1/2} \mathcal{E}_1^{0, \mu}(v,v)^{1/2}} <\infty.
$$
If $\mathcal{E}_1^{0, \mu}$ above can be replaced by $\mathcal{E}^{0, \mu}$ with the same strict inequality, then we say that $(\mathcal{E}, D(L^{\mu}_2))$ satisfies the {\it strong sector condition}. Since $(T^{0, \mu}_t)_{t>0}$ is recurrent but $(T^{\mu}_t)_{t>0}$ is transient, we obtain that $(\mathcal{E}, D(L^{\mu}_2))$ does not satisfy the strong sector condition by \cite[Theorem 1.6.3]{FOT}, \cite[Corollary 14(b)]{GT2} and Proposition \ref{strictirreducibilityprop}. Therefore, in contrast to the case of strong sector condition, we observe that recurrence of $(T^{0, \mu}_t)_{t>0}$ and the weak sector condition of $(\mathcal{E}, D(L^{\mu}_2))$ are not sufficient to obtain recurrence of $(T^{\mu}_t)_{t>0}$.
For another example in this regard, we refer to \cite[Example 25]{GT2}, where $d=1$.
\end{exam}

\begin{exam} \label{Exam:3.3}
\rm
The following gives a simple example for the existence of two distinct finite infinitesimally invariant measures. For another example, we refer to \cite[4.2.1 Example]{BKRS} (see also \cite[Proposition 4.8(i)]{BSR02}).\\[3pt]
Let $A(x)=e^{\|x\|^2} \cdot id$, $x \in \R^d$, $\mathbf{c}$ be a non-zero constant vector field on $\R^d$, $\mathbf{G}(x)= e^{\|x\|^2} \cdot \mathbf{c}$, $\rho(x)=e^{-\|x\|^2}$ and $\widetilde{\rho}(x):=e^{-\|x\|^2+ 2\langle \mathbf{c}, x \rangle}$, $x\ \in \R^d$. Then
$$
Lf(x) = e^{\|x\|^2} \left( \frac12 \Delta f(x) +\langle \mathbf{c}, \nabla f(x) \rangle \right), \quad \text{for all } f \in C_0^{\infty}(\R^d),\; x \in \R^d,
$$
and we can see from Example \ref{exam:3.1} that $\mu=\rho dx$ and $\widetilde{\mu}:=\widetilde{\rho} dx$ are finite infinitesimally invariant measures for $(L, C_0^{\infty}(\R^d))$. By Theorem \ref{preinvlemma}, Proposition \ref{strictirreducibilityprop} and Proposition \ref{fininvariant}, both $(T^{\mu}_t)_{t>0}$ and $(T_t^{\widetilde{\mu}})_{t>0}$ are transient and non-conservative, and both $\mu$ and $\widetilde{\mu}$ are not invariant measures for $(T^{\mu}_t)_{t>0}$ and $(T^{\widetilde{\mu}}_{t})_{t>0}$, respectively. Moreover, by Theorem \ref{pathuniquenesshunt}, $\widetilde{\mu}$ is not an invariant measure for $(T^{\mu}_t)_{t>0}$ and $\mu$ is not an invariant measure for $(T^{\widetilde{\mu}}_t)_{t>0}$.
\end{exam}

\begin{exam} \label{exam:3.4}
\rm
Let $\rho(x)=e^{-\|x\|^2}$, $x \in \R^d$ and $\mu=\rho dx$. Let $w(x):=\int_0^{x_1} e^{s^2} ds$, $x=(x_1, \ldots, x_d) \in \R^d$. Then $w \in C^2(\R^d)$ and
\begin{equation} \label{invariantmeasureequation}
\int_{\R^d} \langle \nabla w, \nabla \varphi  \rangle \rho dx = 0, \quad \text{ for all } \varphi \in C_0^{\infty}(\R^d).
\end{equation}
Let $A=id$, $\mathbf{G}=\frac{\nabla \rho}{2\rho} + \nabla w$. Since \eqref{invariantmeasureequation} holds, $\mu$ is a finite infinitesimally invariant measure for $(L, C_0^{\infty}(\R^d))$, where
$$
L=\frac12 \Delta + \langle \frac{\nabla \rho}{2 \rho}+ \nabla w, \nabla \rangle, \quad \text{ on $C_0^{\infty}(\R^d)$.}
$$
Now let 
$$
\widetilde{\rho}(x):=\rho(x) e^{2w(x)}=\exp\left( -\|x\|^2 +2\int_0^{x_1} e^{s^2} ds  \right), \quad x=(x_1, \ldots, x_d) \in \R^d.
$$ 
Then $\beta^{\widetilde{\rho},A}= \frac{\nabla \rho}{2\rho}+\nabla w$, hence $\widetilde{\mu}:=\widetilde{\rho} dx$ is an infinitesimally invariant (symmetrizing) measure for $(L, C_0^{\infty}(\R^d))$. By Theorem \ref{preinvlemma}, Proposition \ref{strictirreducibilityprop} and Proposition \ref{fininvariant}, $(T^{\mu}_t)_{t>0}$ must be transient and non-conservative. For each $t>0$ and $\mathbf{x}'\in \R^{d-1}$,
$$
\widetilde{\rho}(t, \mathbf{x}') =  \exp\left(-\|\mathbf{x}'\|^2-t^2 +2 \int_0^t e^{s^2} ds\right)  \longrightarrow \infty \quad \text{as }\, t \rightarrow \infty.
$$
Hence by Fubini's Theorem, $\widetilde{\mu}$ is an infinite measure.\\
Finally, by Theorem \ref{preinvlemma}, Proposition \ref{strictirreducibilityprop} and Proposition \ref{fininvariant}, both $(T^{\mu}_t)_{t>0}$ and $(T_t^{\widetilde{\mu}})_{t>0}$ are transient and non-conservative, and both $\mu$ and $\widetilde{\mu}$ are not invariant measures for $(T^{\mu}_t)_{t>0}$ and $(T^{\widetilde{\mu}}_{t})_{t>0}$, respectively. Moreover, by Theorem \ref{pathuniquenesshunt}, $\widetilde{\mu}$ is not an invariant measure for $(T^{\mu}_t)_{t>0}$ and $\mu$ is not an invariant measure for $(T^{\widetilde{\mu}}_t)_{t>0}$. Since $\mu$ is not an invariant measure for $(T^{\mu}_t)_{t>0}$ and $\widetilde{\mu}$ is not an invariant measure for $(T^{\widetilde{\mu}}_t)_{t>0}$, it follows by Remark \ref{equivconinv(iii)} that $(L, C_0^{\infty}(\R^d))$ is not $L^1(\R^d, \mu)$-unique and not $L^1(\R^d, \widetilde{\mu})$-unique. Then, since $\mu$ is finite, we can conclude from \cite[Lemma 1.6(ii)]{Eb} that $(L, C_0^{\infty}(\R^d))$ cannot be $L^r(\R^d, \mu)$-unique for each $r \in (1,2]$.
However, $(L, C_0^{\infty}(\R^d))$ is $L^r(\R^d, \widetilde{\mu})$-unique for each $r \in (1, 2]$
by \cite[Theorem 2.3]{Eb}.

\end{exam}

\begin{exam} \label{dualnotconser}
\rm
Let $i \in \{1, \ldots, d \}$ be fixed. Let $A=id$, and for $x=(x_1, \ldots, x_d) \in \R^d$, $\mathbf{G}(x)=(\frac{1}{2}+ \frac12 e^{-x_i})\mathbf{e}_i$, where $\mathbf{e}_i$ is the $i$-th standard unit vector in $\R^d$ and $\rho(x)=e^{x_i}$. Then {\bf(H)} is satisfied and $\mu=\rho dx$ is an infinitesimally invariant measure for $(L=\frac12 \Delta +\langle \mathbf{G}, \nabla \rangle, C_0^{\infty}(\R^d))$. By \cite[Corollary 2.16(iii)]{LST20}, $(T^{\mu}_t)_{t>0}$ is conservative, hence it follows from Proposition \ref{prop1.22}(ii) that there exists no finite infinitesimally invariant measure for $(L, C_0^{\infty}(\R^d))$. We claim that $(T'^{, \mu}_t)_{t>0}$ is non-conservative (so that $\mu$ is not an invariant measure for $(T^{\mu}_t)_{t>0}$ by Lemma \ref{equicondi32}(i)). To see the latter, consider a positive function $\Psi \in C_b^2((0, \infty))$ which is non-zero on $(0, \infty)$ and let $V(x)=\Psi(e^{-x_i})$, $x =(x_1,\ldots,x_d ) \in \R^d$. Then $V \in C_b^2(\R^d)$ is positive and a non-zero function. If there exists a constant $\alpha>0$ such that 
\begin{equation} \label{non-conservativecriteriast}
L'^{, \mu}V \geq \alpha V \;\;\; \text{ on $\R^d$},
\end{equation}
then what we claimed above follows from \cite[Proposition 2.18]{LST20}.
Note that
\begin{equation} \label{fundinequ}
L'^{, \mu}V= \frac12 \partial_{ii} V +(\frac12-\frac12 e^{-x_i} )  \partial_i V =\frac12 \left( \Psi''(e^{-x_i})+\Psi'(e^{-x_i}) \right) \cdot \left(e^{-x_i}\right)^2.
\end{equation}
Thus, in order to show \eqref{non-conservativecriteriast}, it suffices to find $\alpha>0$, such that 
$$
\left(\Psi''(y)+\Psi'(y) \right) y^2 \geq 2\alpha \Psi(y), \;\; \text{ for all } y>0.
$$
Define
$$
\Psi(y):=\left\{\begin{matrix}
 y^2(6-y) \quad &\text{ if }& 0 < y \leq 3, \\ 
54-\frac{81}{y} \quad &\text{ if }& 3 \leq y
\end{matrix}\right.
$$
Then $\Psi \in C_b^2((0, \infty))$ and as in \cite[Remark 2.19(ii)]{LST20}, it is easy to see that \eqref{fundinequ} holds with $\alpha =\frac{1}{4}$, so that the claim is proved.
\text{}\\
Let $\widetilde{\rho}(x):=e^{x_i-e^{-x_i}}$, $x=(x_1,\ldots,x_d)\in \R^d$. Then $\beta^{\widetilde{\rho}, id}=(\frac12 +\frac12 e^{-x_i}) \mathbf{e}_i=\mathbf{G}(x)$, hence $\widetilde{\mu}:=\widetilde{\rho} dx$ is another infinitesimally invariant symmetrizing) measure for $(L, C_0^{\infty}(\R^d))$. 
Thus by Proposition \ref{strictirreducibilityprop} and Theorem \ref{preinvlemma} $(T^{\mu}_t)_{t>0}$ must be transient, and then $(T'^{, \mu}_t)_{t>0}$ and $(T^{\widetilde{\mu}}_t)_{t>0}$ are also transient by Proposition \ref{strictirreducibilityprop} and Theorems \ref{cosemirec} and \ref{pathuniquenesshunt}.
By Theorem \ref{pathuniquenesshunt}, 
\begin{equation} \label{sameidenti}
T^{\mu}_t f=T^{\widetilde{\mu}}_t f  \quad \text{a.e.} \;\text{ for all } f \in  L^1(\R^d, \mu) \cap L^1(\R^d, \widetilde{\mu}) \text{ and } t>0.
\end{equation}
Thus since $(T^{\mu}_t)_{t>0}$ is conservative, so is $(T^{\widetilde{\mu}}_t)_{t>0}$. Since $(T^{\widetilde{\mu}}_t)$ is associated with the symmetric Dirichlet form $(\mathcal{E}^{0, \widetilde{\mu}}, D(\mathcal{E}^{0, \widetilde{\mu}}))$, we obtain $(T^{\widetilde{\mu}}_t)_{t>0}=(T'^{, \widetilde{\mu}}_t)_{t>0}$ on  $L^2(\R^d, \widetilde{\mu})$, hence $\widetilde{\mu}$ is an invariant measure for $(T^{\widetilde{\mu}}_t)_{t>0}$ by Lemma \ref{equicondi32}(i), so that $\widetilde{\mu}$ is an invariant measure for $(T^{\mu}_t)_{t>0}$ by \eqref{sameidenti}. By Theorem \ref{theo1.12}, we obtain a bounded co-excessive function $\frac{\widetilde{\rho}}{\rho}=e^{-e^{-x_i}} \in C^{\infty}(\R^d) \cap \mathcal{B}_b(\R^d)$ for $(T^{\mu}_t)_{t>0}$, i.e.
\begin{equation} \label{excessivefun}
T'^{, \mu}_t\left(\frac{\widetilde{\rho}}{\rho}\right)  =\frac{\widetilde{\rho}}{\rho}, \;\; \text{$\mu$-a.e.} \;\; \forall t>0.
\end{equation}
We finally mention that by Remark \ref{equivconinv(iii)},  $(L, C_0^{\infty}(\R^d))$ is $L^1(\R^d, \widetilde{\mu})$-unique but not $L^1(\R^d, \mu)$-unique. In other words, $(T^{\widetilde{\mu}}_t)_{t>0}$ is the unique $C_0$-semigroup on $L^1(\R^d, \widetilde{\mu})$ whose generator extends $(L, C_0^{\infty}(\R^d))$, but there exists a $C_0$-semigroup $(\mathcal{T}^{\mu}_t)_{t>0}$ on $L^1(\R^d, \mu)$ whose generator extends $(L, C_0^{\infty}(\R^d))$ and $h \in \mathcal{B}_b(\R^d)_0 \subset L^1(\R^d, \mu) \cap L^1(\R^d, \widetilde{\mu})$, $t_0>0$ such that
$$
\mathcal{T}^{\mu}_{t_0} h \neq T^{\mu}_{t_0} h = T^{\widetilde{\mu}}_{t_0} h, \quad \text{$\mu$-a.e.}
$$
Here we emphasize the following: it follows from \cite[Proposition 3.51, Theorem 3.52]{LST20} that pathwise uniqueness and strong existence holds for the following It\^{o}-SDE related to $(L, C_0^{\infty}(\R^d))$
\begin{equation} \label{sdebrown}
X_t =x+W_t + \int_0^t \mathbf{G}(X_s)ds,  \quad 0 \leq t < \infty, \ \ x\in \R^d
\end{equation}
and from \cite[Theorem 3.52]{LST20} that $(T^{\mu}_t)_{t>0}$ is a version (on $\mathcal{B}_b(\R^d))$ of the transition semigroup of the pathwise unique and strong solution to \eqref{sdebrown}.  But $L^1(\R^d, \mu)$-uniqueness of $(L, C_0^{\infty}(\R^d))$ does not hold. 
\end{exam}

\begin{exam} \label{appstaeple} \rm
For a fixed $i \in \{1, \ldots, d \}$, let $A=id$, $\mathbf{G}(x)=(\frac{1}{2}- \frac12 e^{-x_i})\mathbf{e}_i$ and $\rho(x)=e^{x_i}$, $x=(x_1, \ldots, x_d) \in \R^d$. Then $\mu=\rho dx$ is an infinitesimally invariant measure for $(L, C_0^{\infty}(\R^d))$. Since here $(T^{\mu}_t)_{t>0}$ coincides with $(T'^{, \mu}_t)_{t>0}$ of Example \ref{dualnotconser}, $(T^{\mu}_t)_{t>0}$ is transient and non-conservative and $(T'^{, \mu}_t)_{t>0}$ is transient and conservative, hence $\mu$ is an invariant measure for $(T^{\mu}_t)_{t>0}$ by Lemma \ref{equicondi32}(i). Let $\widetilde{\rho}(x)=e^{x_i+e^{-x_i}}$, $x=(x_1,\ldots,x_d)\in \R^d$. Then $\beta^{\widetilde{\rho}, id}=(\frac12 -\frac12 e^{-x_i})\mathbf{e}_i$, hence $\widetilde{\mu}=\widetilde{\rho} dx$ is an infinitesimally invariant symmetrizing measure  for $(L, C_0^{\infty}(\R^d))$.  By Theorem \ref{pathuniquenesshunt}, 
\begin{equation} \label{sameidenti2}
T^{\mu}_t f=T^{\widetilde{\mu}}_t f  \quad \text{a.e.} \;\text{ for all } f \in  L^1(\R^d, dx) \cap L^1(\R^d, \widetilde{\mu}) \text{ and } t>0,
\end{equation}
hence $(T^{\widetilde{\mu}}_t)_{t>0}$ is non-conservative. Since $(T_t^{\widetilde{\mu}})_{t>0}$ is associated to the symmetric Dirichlet form ($\mathcal{E}^{0, \widetilde{\mu}}, D(\mathcal{E}^{0, \widetilde{\mu}}))$,  we obtain $(T^{\widetilde{\mu}}_t)_{t>0}=(T'^{, \widetilde{\mu}}_t)_{t>0}$ on  $L^2(\R^d, \widetilde{\mu})$, so that $\widetilde{\mu}$ is not an invariant measure for $(T^{\widetilde{\mu}}_t)_{t>0}$ by Lemma \ref{equicondi32}(i).  By \eqref{sameidenti2} again, $\widetilde{\mu}$ is not an invariant measure but a sub-invariant measure for $(T^{\mu}_t)_{t>0}$. In particular, by Theorem \ref{theo1.12} and \eqref{sameidenti2}, there exists an excessive function $\frac{\rho}{\widetilde{\rho}}= e^{-e^{-x_i}} \in C^{\infty}(\R^d) \cap \mathcal{B}_b(\R^d)$ for $(T^{\widetilde{\mu}}_t)_{t>0}$ and $(T^{\mu}_t)_{t>0}$ such that
$$
T^{\mu}_t \left( \frac{\rho}{\widetilde{\rho}} \right)=T^{\widetilde{\mu}}_t \left( \frac{\rho}{\widetilde{\rho}} \right) =T'^{,\widetilde{\mu}}_t \left( \frac{\rho}{\widetilde{\rho}} \right) =\frac{\rho}{\widetilde{\rho}}, \quad \; \forall t>0,
$$
which is ultimately the same as the one in \eqref{excessivefun}. By Remark \ref{equivconinv(iii)}, $(L, C_0^{\infty}(\R^d))$ is $L^1(\R^d, \mu)$-unique, but not $L^1(\R^d, \widetilde{\mu})$-unique, hence there exists a $C_0$-semigroup $(\mathcal{T}^{\widetilde{\mu}}_t)_{t>0}$ on $L^1(\R^d, \widetilde{\mu})$ whose generator extends $(L, C_0^{\infty}(\R^d))$ and $h \in \mathcal{B}_b(\R^d)_0 \subset L^1(\R^d, \widetilde{\mu}) \cap L^1(\R^d, \widetilde{\mu})$, $t_0>0$ such that
$$
\mathcal{T}^{\widetilde{\mu}}_{t_0} h \neq T^{\widetilde{\mu}}_{t_0} h= T^{\mu}_{t_0} h.
$$
By \cite[Theorem 2.3 and Lemma 1.6(i)]{Eb}, $(L, C_0^{\infty}(\R^d))$ is $L^r(\R^d, \widetilde{\mu})$-unique for all $r \in (1,2]$, hence strongly Markov unique, i.e. $(T^{\widetilde{\mu}}_t)_{t>0}$ is the unique sub-Markovian $C_0$-semigroup on $L^1(\R^d, \widetilde{\mu})$ whose generator extends $(L, C_0^{\infty}(\R^d))$. Therefore $(\mathcal{T}^{\widetilde{\mu}}_t)_{t>0}$ can not be sub-Markovian.
\end{exam}

\text{}\\
\subsection{Applications} \label{4.2appli}
Here, we introduce some applications of our results. Under the assumption {\bf (H)}, let $\mu$ be an infinitesimally invariant measure for $(L, C_0^{\infty}(\R^d))$. Then by \cite[Proposition 3.10]{LT18} (see also \cite[Theorem 3.11]{LST20}), there exists a regularized semigroup $(P^{\mu}_t)_{t>0}$ such that for any $f \in \cup_{r \in [1, \infty]} L^r(\R^d, \mu)$ and $t>0$
$$
P^{\mu}_{\cdot} f \in C(\R^d \times (0, \infty)) \quad \text{ and }\quad  P^{\mu}_t f = T^{\mu}_t f  \ \text{ $\mu$-a.e.}
$$
Moreover, by \cite[Theorem 3.12]{LT18} there exists a Hunt process
$$
\M=(\Omega, \mathcal{F}, (\mathcal{F}_t)_{t \geq 0}, (X_t)_{t \geq 0}, (\P_x)_{x \in \R^d \cup \Delta})
$$
such that $(P^{\mu}_t)_{t>0}$ is the transition semigroup of $\M$, i.e.
\begin{equation} \label{consemi}
P^{\mu}_t f(x)  = \E_x[f(X_t)] \quad \text{ for all }  \ f \in \mathcal{B}_b(\R^d), \ x \in \R^d, \ t>0.
\end{equation}
Now assume that $(T^{\mu}_t)_{t>0}$ is conservative. Then $\M$ is non-explosive by \cite[Corollary 3.23]{LST20}, hence by \cite[Theorem 3.19(i)]{LT18} (see also \cite[Theorem 3.22(i)]{LST20}), $\M$ weakly solves $\P_x$-a.s. for any $x \in \R^d$ the following It\^{o}-SDE 
\begin{equation} \label{weaksolusde}
X_{t} = x+\int_0^{t} \sigma(X_s) dW_s +\int_0^{t} \mathbf{G}(X_s) ds, \quad 0\leq t<\infty,
\end{equation}
where $\sigma=(\sigma_{ij})_{1 \leq i,j \leq d}$ is an arbitrary matrix of continuous functions with $A=\sigma \sigma^T$ and $(W_t)_{t \geq 0}$ is a standard $(\mathcal{F}_t)_{t \geq 0}$-Brownian motion on $(\Omega,\mathcal{F}, \P_x )$. Additionally, if $\sigma_{ij} \in H^{1,p}_{loc}(\R^d)$ for all $1 \leq i,j \leq d$, then by \cite[Theorem 3.52]{LST20}
 any weak solution to \eqref{weaksolusde} has the same law as $\P_x \circ X^{-1}$, hence the transition semigroup generated by an arbitrary weak solution to \eqref{weaksolusde} for all $x \in \R^d$ coincides with  $(P^{\mu}_t)_{t>0}$ on $\mathcal{B}_b(\R^d)$ and therefore inherits all properties from 
$(T^{\mu}_t)_{t>0}$ developed in this article and \cite{LST20}.
\\
\text{}\\
In \cite[Chapter 2.2]{LB07} and \cite[Section 4]{MPW02}, without using the existence of an infinitesimally invariant measure for $(L, C_0^{\infty}(\R^d))$ as we did, and under the assumption that the coefficients of $L$ are locally H\"{o}lder continuous on $\R^d$ and $A=(a_{ij})_{1 \leq i,j \leq d}$ is locally uniformly strictly elliptic (cf. \eqref{use}), a sub-Markovian strong Feller semigroup $(T(t))_{t>0}$ on $\mathcal{B}_b(\R^d)$ is constructed, and given $f \in C_b(\R^d)$ it holds that $u_f:=T(\cdot) f \in C_b(\R^d \times [0, \infty)) \cap C^{2+\alpha, 1+\frac{\alpha}{2}}(\R^d \times (0, \infty))$ and that $u_f$ solves the following Cauchy problem pointwisely
\begin{equation} \label{solcauchypro}
\partial_t u_f   = L u_f =\frac12 \sum_{i,j=1}^d a_{ij}\partial_{ij} u_f+\sum_{i=1}^d g_i \partial_i u_f\;\; \text{ in $\R^d \times (0, \infty)$},\quad u_f(\cdot, 0)=f  \text{ in $\R^d$}.
\end{equation}

\begin{prop} \label{equivsemi}
Assume {\bf (H)} and that the components of $\mathbf{G}$ are locally H\"{o}lder continuous. If $(T^{\mu}_t)_{t>0}$ is conservative, then $(T(t))_{t>0} =(P^{\mu}_t)_{t>0}$ on $\mathcal{B}_b(\R^d)$ where $(T(t))_{t>0}$ 
and $(P^{\mu}_t)_{t>0}$ are defined right above.
\end{prop}
\begin{proof}
Let $x \in \R^d$ and $n \in \N$ with $n \geq \|x\|+1$. Define $\sigma_n:=\inf \{t>0: X_t \in \R^d \setminus B_n \}$. Given $i \in \{1, \ldots, d\}$, let $X^{i, n}_t:= X^i_{t \wedge \sigma_n}$ where $X_t=(X^1_t, \ldots, X^d_t)$. From \eqref{weaksolusde}, for any $i,j \in \{1, \ldots, d \}$ it follows
$$
\langle X^{i,n}, X^{j,n} \rangle_t  = \int_0^{t \wedge \sigma_n} a_{ij}(X_s) ds, \quad 0\leq t< \infty, \;\; \text{$\P_x$-a.s.}
$$
Given $f \in C_0(\R^d)$ and $t_0>0$, define
$$
v_f(y,t):=T(t_0-t)f(y), \quad y \in \R^d, \;0 \leq t \leq t_0.
$$
Then by \eqref{solcauchypro}, $v_f \in C_b(\R^d \times [0, t_0]) \cap C^{2,1}(\R^d \times [0, t_0))$, $v_f(y,t_0)=f(y)$ for any $y \in \R^d$ and
\begin{equation*}
\partial_t v_f + Lv_f = 0 \quad \text{ in $\R^d \times (0, t_0)$}.
\end{equation*}
By the time-dependent It\^{o}-formula and \eqref{weaksolusde}, for any $0\leq t< t_0$,
$$
v_f(X_{t \wedge \sigma_n}, t \wedge \sigma_n) -v_f(x,0) = \int_0^{t \wedge \sigma_n} \nabla v_f(X_s,s) \sigma(X_s) dW_s +\int_0^{t \wedge \sigma_n} \left(\partial_t v_f + Lv_f \right) (X_s, s) ds, \quad \text{$\P_x$-a.s.}
$$
Thus, for any $0 \leq t<t_0$,
$$
\E_x \left[ v_f (X_{t \wedge \sigma_n}, t \wedge \sigma_n)\right]  = \E_x\left[ v_f(x, 0) \right] =T(t_0) f(x)
$$
Letting $t \rightarrow t_0-$, 
$$
T(t_0)f(x)=\E_x \left[ v_f(X_{t_0 \wedge \sigma_n}, t_0 \wedge \sigma_n) \right].
$$
Since $\M$ is non-explosive, we have by \cite[Lemma 3.17(i), Definition 3.21]{LST20} $\P_x(\lim_{n\to \infty}\sigma_n=\infty)=1$. Thus, it holds
\begin{eqnarray*}
T(t_0)f(x)&=&\E_x \left[ v_f(X_{t_0}, t_0) \right] = \E_x \left[f(X_{t_0}) \right] =P^{\mu}_{t_0} f(x).
\end{eqnarray*}
Since $x \in \R^d$,  $t_0>0$ are arbitrary and the $\sigma$-algebra generated by $C_0(\R^d)$ equals $\mathcal{B}(\R^d)$, the assertion follows from a monotone class argument.
\end{proof}

\begin{rem} \rm
\begin{itemize}
\item[(i)]
In \cite{LST20} diverse explicit conditions on the coefficients $A$ and $\mathbf{G}$ to obtain conservativeness of $(T^{\mu}_t)_{t>0}$ are presented. For instance, by \cite[Corollaries 3.23 and 3.27]{LST20}, if there exist constants $M>0$ and  $N_0 \in \N$ such that
\begin{eqnarray} \label{nonexpcondi}
-\frac{\langle A(x)x, x \rangle}{ \left \| x \right \|^2 }+ \frac12\mathrm{trace}A(x)+ \big \langle \mathbf{G}(x), x \big \rangle \leq M\left  \| x \right \|^2 \big( \ln \left \| x \right \| +1      \big)
\end{eqnarray}
for a.e. $x\in \R^d\setminus \overline{B}_{N_0}$, then $(T^{\mu}_t)_{t>0}$ is conservative. 
Under the assumptions of Proposition \ref{equivsemi}, the results in this article and \cite{LST20} apply to the sub-Markovian strong Feller semigroup $(T(t))_{t>0}$ constructed in \cite[Chapter 2.2]{LB07} (or \cite[Section 4]{MPW02}). More precisely, not only our semigroup $(T^{\mu}_t)_{t>0}$ has all properties inherited from $(T(t))_{t>0}$ in \cite{LB07} (or \cite{MPW02}) if the components of our $\mathbf{G}$ are locally H\"{o}lder continuous, but also $(T(t))_{t>0}$ has new properties inherited from $(T^{\mu}_t)_{t>0}$ in this article and \cite{LST20} if $A$ in \cite{LB07} (or \cite{MPW02}) has components which are in $H^{1,p}_{loc}(\R^d)$ for some $p>d$.
 For instance, in \cite[Chapter 8]{LB07} (or \cite[Section 6]{MPW02}), existence of an invariant probability measure is derived under certain conditions using the Krylov-Bogoliubov Theorem (\cite[Theorem 8.1.19]{LB07}), while existence, non-existence, uniqueness and non-uniqueness of invariant measures for $(T(t))_{t>0}$ can be studied through the properties of $(T^{\mu}_t)_{t>0}$ and under the assumptions of Proposition \ref{equivsemi}.
\item[(ii)]
Assume that
$$
Lf = \text{trace}(Q \nabla^2 f) + \langle Bx, \nabla f \rangle, \quad  f\in C_0^{\infty}(\R^d)
$$
where $Q=(q_{ij})_{1 \leq i,j \leq d}$ is a symmetric and strictly positive definite matrix with constant $q_{ij} \in \R$ for all $1 \leq i,j \leq d$  and  $B=(b_{ij})_{1 \leq i,j \leq d}$ is a non-zero matrix with constant $b_{ij}\in \R$ for all $1 \leq i,j \leq d$. 
The transition semigroup $(P^{\mu}_t)_{t>0}$ of $\M$ (which is a regularized $\mu$-version of $(T^{\mu}_t)_{t>0}$)
can be identified as the Ornstein--Uhlenbeck semigroup $(T(t))_{t>0}$ of \cite[(2.4)]{LMP20} (or \cite[(9.1.5)]{MPW02}), where for $t>0$ and $x \in \R^d$
\begin{equation*} 
T(t)f(x) := \frac{1}{{(4\pi)}^{d/2} \text{det}(Q_t)^{1/2}}
\int_{\R^d} e^{-\langle Q_t^{-1}y, y\rangle/4} f(e^{tB}x-y) dy
\end{equation*}
and $Q_t:=\int_0^t e^{\tau B} Q e^{\tau B^T} d\tau$. Indeed, for a matrix $\sigma=(\sigma_{ij})_{1 \leq i,j \leq d}$ satisfying $\frac12 \sigma \sigma^T=Q$ with constant $\sigma_{ij} \in \R$ for all $1 \leq i,j \leq d$ and for a standard Brownian motion $(\widetilde{W}_t)_{t \geq 0}$ on $(\widetilde{\Omega},\widetilde{\mathcal{F}}, \widetilde{\P} )$, the unique strong solution $(Y^x_t)_{t \geq 0}$ to the following SDE with the initial condition $x \in \R^d$
$$ 
Y^x_t = x + \sigma \widetilde{W}_t + \int_0^t BY^x_s ds, \quad 0 \leq t <\infty, \; \; \widetilde{\P} \text{-a.s.}
$$
has the closed form 
$$ 
Y^x_t = e^{tB} x +\int_0^t e^{(t-s) B} \sigma d\widetilde{W}_s, \quad 0 \leq t<\infty.
$$
Then $A:=2Q$ and $\mathbf{G}(x):=Bx$ satisfy  assumption {\bf(H)} and \eqref{nonexpcondi} and it is known that
$$  
T(t) f(x)  = \widetilde{\E}[f(Y^x_t)], \quad \forall x \in \R^d, f \in C_b(\R^d),
$$
where $\widetilde{\E}$ is the expectation with respect to $\widetilde{\P}$ (see \cite[Section 2]{LMP20} or \cite[Section 9.1]{MPW02}).  On the other hand, by \eqref{consemi} and \cite[Theorem 3.52]{LST20} 
$$
P^{\mu}_tf(x)  = \E_x[f(X_t)]= \widetilde{\E}[f(Y^x_t)], \quad \forall x \in \R^d, f \in C_b(\R^d),
$$
Hence as in (i), $(P^{\mu}_t)_{t>0}$ shares various properties with $(T(t))_{t>0}$. 
\end{itemize}
\end{rem}

\centerline{}
Haesung Lee, Gerald Trutnau\\
Department of Mathematical Sciences and \\
Research Institute of Mathematics of Seoul National University,\\
1 Gwanak-Ro, Gwanak-Gu,
Seoul 08826, South Korea,  \\
E-mail: fthslt14@gmail.com, trutnau@snu.ac.kr
\end{document}